\documentclass[11pt]{article}
\usepackage[utf8]{inputenc}
\usepackage{amsmath, amsthm, latexsym, amssymb}
\usepackage{amsfonts}
\usepackage{bbm}
\usepackage{bm}
\usepackage{cases}
\usepackage{mathtools}
\usepackage{graphicx}
\usepackage{listings}
\usepackage{hyperref} 
\usepackage{xcolor}
\usepackage[margin=1.0in]{geometry}
\usepackage{lineno}
\usepackage{todonotes}
\usepackage{upgreek}
\usepackage{enumerate}
\usepackage{fonts, defs}
\usepackage{comment}

\newtheorem{thm}{Theorem}

\newtheorem{lem}[thm]{Lemma}
\newtheorem*{lem*}{Lemma}
\newtheorem{corl}{Corollary}
\newtheorem{defn}{Definition}

\title{Numerical Analysis of a Hybrid Method for Radiation Transport%
\thanks{This material is based upon work supported by the U.S. Department of Energy, Office of Science, Office of Advanced Scientific Computing Research, as part of their Applied Mathematics Research Program. The work was performed at the Oak Ridge National Laboratory, which is managed by UT-Battelle, LLC under Contract No. De-AC05-00OR22725. The United States Government retains and the publisher, by accepting the article for publication, acknowledges that the United States Government retains a non-exclusive, paid-up, irrevocable, world-wide license to publish or reproduce the published form of this manuscript, or allow others to do so, for the United States Government purposes. The Department of Energy will provide public access to these results of federally sponsored research in accordance with the DOE Public Access Plan (http://energy.gov/downloads/doe-public-access-plan).
}} 
\author{Cory, Victor, }
\date{\today}                                                                                              \author{Andr\'es Galindo-Olarte\thanks{Department of Mathematics, Michigan State University, East Lansing, MI 48824, USA 
}
\and Victor P.\ DeCaria\thanks{Computer Science and Mathematics Division, Oak Ridge National Laboratory, Oak Ridge, TN 37831, USA
}
\and Cory D.\ Hauck\thanks{Computer Science and Mathematics Division, Oak Ridge National Laboratory, Oak Ridge, TN 37831, USA and 
Mathematics Department, University of Tennessee, Knoxville, TN 37996, USA 
} 
}

\begin{document}
\maketitle
\begin{abstract}
    In this work, we prove rigorous error estimates for a hybrid method introduced in \cite{hauck2013collision} for solving the time-dependent radiation transport equation (RTE).  The method relies on a splitting of the kinetic distribution  function for the radiation into uncollided and collided components. A high-resolution method (in angle) is used to approximate the uncollided components and a low-resolution method is used to approximate the the collided component. After each time step, the kinetic distribution is reinitialized to be entirely uncollided.  For this analysis, we consider a mono-energetic problem on a periodic domains, with constant material cross-sections of arbitrary size.  To focus the analysis, we assume the uncollided equation is solved exactly and the collided part is approximated in angle via a spherical harmonic expansion ($\Peqn$ method). Using a non-standard set of semi-norms, we obtain estimates of the form $C(\e,\sigma,\dt)N^{-s}$ where $s\geq 1$ denotes the regularity of the solution in angle, $\e$ and $\sigma$ are scattering parameters, $\dt$ is the time-step before reinitialization, and $C$ is a complicated function of $\e$, $\sigma$, and $\dt$.  These estimates involve analysis of the multiscale RTE that includes, but necessarily goes beyond, usual spectral analysis. We also compute error estimates for the monolithic $\Peqn$ method with the same resolution as the collided part in the hybrid. Our results highlight the benefits of the hybrid approach over the monolithic discretization in both highly scattering and streaming regimes.
\end{abstract}


\section{Introduction}
The radiation transport equation (RTE) \cite{pomraning2005equations,lewis1984computational,case1967linear} describes the movement of particles through a material medium by means of a kinetic distribution function that gives the density of particles with respect to the local phase space measure.  In a general setting, the phase space is six dimensional:  three dimensions for particle position and three for particle momentum, the latter of which is typically decomposed into energy and direction (or angle) of flight.  Thus in the time-dependent setting, the RTE is defined over a seven-dimensional domain.

The RTE describes two basic processes: particle advection and interactions with the material medium.   These interactions can be of various types and include scattering and emission/absorption processes.  The rate at which these processes occur is determined by the properties of the material, expressed via cross-sections.  Material cross-sections may vary in space and depend on the particle energy and, in situations that the material evolves, the cross-sections may evolve as well.  When cross-sections vary significantly, the RTE may exhibit multiscale behavior.  It is the combination of this multiscale behavior with the high-dimensional phase space that makes simulating the RTE a challenging task.

A well-known multi-scale feature of the RTE is the diffusion limit.  In regions where the scattering cross-section is large, the solution of the RTE can be accurately approximated by its angular average \cite{larsen1974asymptotic,bensoussan1979boundary}. Moreover this average is well-approximated by the solution of a diffusion equation.  This solution to diffusion equation has long been used as a cheap approximation the solution to the RTE in scattering dominated regimes.

Another common limit is the absorption limit, which is characterized by a complete lack of scattering.  In this case, the RTE does not have a simple asymptotic approximation.  However, due the abscence of scattering, there is no coupling between the angles and energies of the kinetic distribution.  Thus, with a proper discretiation, the RTE solution can be easily parallelized.

In problems for which the scattering cross-section varies dramatically, both of the limits above can exist simultaneously, along with a range of transition regimes in between. A consequence of this fact is that a monolithic numerical treat of the RTE will require many degrees of freedom that are strongly coupled.  In practice, the time-dependent RTE is often updated in time with an implicit scheme.  In such cases, designing the linear solvers can be a challenge.

A variety of approaches have been proposed for addressing the multiscale challenges posed of the RTE.  These include micro-macro decompositions \cite{lemou2008new}, high order-low order (HOLO) methods \cite{chacon2017multiscale}, diffusion-based acceleration \cite{lewis1984computational, adams2002fast}, and preconditioned Krylov approaches \cite{warsa2004krylov}.  In the current paper, we consider a hybrid formulation \cite{hauck2013collision} that is based on the notion of first-collision source \cite{alcouffe2006first}.  In this hybrid formulation, the RTE is split into two components:  an uncollided component that tracks the particles up to point of their first material interaction and a collided component that track the particles that remain.  The resulting system is then approximated with two different angular discretizations:  a high-resolution discretization for the uncollided equation and a low-resolution for the collided equation.  The intuition that drives this strategy is that scattering produces a smoother solution; hence the collided equation should require less resolution to recover an accurate solution.  The uncollided equation, on the other, requires higher resolution; however it takes the form of a purely absorbing RTE and can therefore be solved much more efficiently the original RTE using the same number of degrees of freedom.   The efficiency of the hybrid approach for the RTE has been demonstrated in several papers \cite{crockatt2017arbitrary, crockatt2019hybrid, crockatt2020improvements}, including generalizations to hybrid energy discretizations \cite{Whewell2022} and hybrid spatial discretizations \cite{heningburg2020hybrid}.   

A key component of the hybrid implementation for the time-dependent RTE is a relabeling procedure that, after a given time step, maps the collided numerical solution into the space of the uncollided numerical solution and then uses the sum to re-initialize the uncollided equation.  Meanwhile, the collided equation is re-initialized to zero.  This relabeling step is critical, since otherwise the hybrid numerical solution would eventually convergence to a low-resolution numerical solution of the collided equation.

Despite the intuitive motivation of the hybrid and the success of the hybrid approach in numerical simulations, the method still lacks rigorous justification.  This is due in part to complications introduced by the multiscale behavior of the RTE.  For example, spectral approximations of the RTE in angle are fairly straightforward to analyze \cite{frank2016convergence}, but a multiscale analysis that takes into account the degree of scattering is significantly more complicated \cite{chen2019multiscale}.  The relabeling step of the hybrid formulation complicates the situation even further.

In the current paper, we take a first step in analyzing the hybrid method for the time-dependent mono-energetic version of the RTE with isotropic scattering.  We focus only on the angular discretization of the RTE, comparing the standard spectral approximation ($\Peqn$) for the full system with a discretization of the hybrid that features a spectral approximation of the same resolution for the collided equation but assumes an exact solution for the uncollided equation.  Clearly, the hybrid formulated in this way is more expensive than the monolithic approach.  Thus the goal of the analysis is determine what is gained from the extra work involves in a high-resolution simulation of the uncollided equation, which in practice is computed with a high-fidelity collocation method or with a Monte-Carlo method.

The remainder of this paper is organized as follows.  In Section 2, we introduce the RTE, reduce it to the purely scattering problem, recall the $\Peqn$ method, and then describe the setup of the hybrid.  Having established the setting of of the problem, we then summarize the main results of the paper.  In Section 3, we derive error estimates for the $\Peqn$ equations.  In Section 4, we analyze the hybrid problem.   In Section 5, we generalize results back to the original RTE with non-zero absorption.  In Section 6, we provide a short conclusion.  The appendix contains some generic results used for the estimates of the main paper.

\section{Background}
\label{sec:background}
\subsection{The radiation transport equation}
We consider a time-dependent transport equation with periodic boundaries, isotropic scattering, unit-speed particles, and diffusion scaling:
 \begin{subequations}   \begin{align}
       &\e\partial_t \Psi^{\veps}+\Omega\cdot\nabla_x \Psi^{\veps}+\frac{\sigt}{\e}\Psi^{\veps}=\left(\frac{\sigt}{\e} - \e\siga\right)\overline{\Psi^{\veps}}+\e Q, \qquad \overline{\Psi^{\veps}}=\frac{1}{4\pi}\int_{\sphe}\Psi^{\veps}\,d\Omega, \\
        &\left.\Psi^{\veps}\right|_{t=0}=g.
    \end{align}    	\label{eq:transp}%
\end{subequations}
Here $\Psi^{\veps}= \Psi^{\veps}(x,\Omega,t)$ is a function of position $x\in X=[0,2\pi)^3$, direction of flight $\Omega\in\sphe$, and time $t>0$.  It can be interpreted physically as the density of particles at time $t$ with respect to the measure $d\Omega dx$.  Particles interact with a material background characterized by an absorption cross-section $\siga \geq 0$, total cross-section $\sigt \geq \siga$ (which accounts for scattering and absorption), and a known source $Q=Q(x,\Omega,t)$. The quantity $\sigt - \veps^2 \siga$ is the scaled scattering cross-section, where the non-dimensional parameter $\varepsilon > 0$ characterizes the strength of the scattering as well as the relevant time scale.  Indeed, it is well-known \cite{larsen1974asymptotic} that in the limit $\veps \to 0$, $\Psi^{\veps} \to \Psi^0$ where $\Psi^0$ is independent of angle and satisfies the diffusion equation of the form
 \begin{subequations}   
 \begin{align}
       &\p_t \Psi^0 - \nabla_x \cdot \left( \frac{1}{3\sigt}  \nabla_x \Psi^0 \right) + \siga \Psi^0 =0, \qquad \left.\Psi^0\right|_{t=0}= \frac{1}{4\pi}\int_{\sphe}g\,d\Omega .
    	\label{eq:diff_limit}
    \end{align}
    \end{subequations}
For $g \in L^2(X \times \sphe)$, $\sigt, \siga \in L^\infty(X)$, $Q \in L^2(X \times \bbS^2 \times [0, \infty))$ \eqref{eq:transp}, is known to have a semi-group solution $\Psi^{\veps} \in C([0,\infty); L^2(X \times \sphe) )$ \cite[Theorem XXI.2.3]{dautray1999mathematical}. If in addition, $\Omega \cdot \nabla_x g \in L^2(X \times \sphe)$, then $\Psi^{\veps} \in C^1([0,\infty); L^2(X \times \sphe) )$.  We assume this is the case for remainder of the paper.

In order to facilitate a clear stability and error analysis, we assume that the cross-sections $\siga$ and $\sigt$ are constant in space. 
This assumption on $\siga$ allows us to convert \eqref{eq:transp} to a purely scattering system for the function $\psi = e^{\siga t}\Psi^{\veps}$:%
\footnote{To reduce notation, we suppress the dependence of $\psi$ on $\veps$.}
\begin{subequations}
\label{eq:transp_simple}
\begin{align}
&\e\partial_t \psi + \Omega \cdot \nabla_x \psi + \frac{\sigma}{\e}\psi = \frac{\sigma}{\e}\overline{\psi} + \e q
, \qquad  \overline{\psi}=\frac{1}{4\pi}\int_{\sphe}\psi\,d\Omega, \\
        &\left.\psi\right|_{t=0}=g,
    \end{align}
\end{subequations}
where $q = e^{\siga t}Q$ and $\sigma := \sigt$.
 Henceforth, we focus our analysis on \eqref{eq:transp_simple}. The results can then be translated back to the case of non-zero absorption by undoing the transformation, which gives exponential decay if $\siga > 0$.  This assumption is made for simplicity, but it does introduce a measure of regularity into the solution that is not typical in applications.  Indeed, a more reasonable assumption is that the cross-sections are piece-wise smooth and that the boundaries are equipped with inflow data.  Hence the analysis here can be viewed as a localized proxy for a more realistic scenario. A more sophisticated analysis to include boundary and interior layers is the subject of future work.

\subsection{The \texorpdfstring{$\Peqn$}{PN} approximation}

Given $N\in\mathbb{N}_{\geq 0}$, the $\Peqn$ method is a spectral discretization of the transport equation with respect to the angular variable $\Omega$.  Let $\{m_{\ell,k}\}_{\ell,k}$ be the real-valued, orthonormal basis of spherical harmonics, where $\ell \geq 0$ denotes the degree and $k \in\{ -\ell , \dots, \ell\}$ denotes the order. For any $u \in L^2(\sphe)$, the angular moment $u_{\ell,k}$ is given by
\begin{equation}
    u_{\ell,k} = \int_{\sphe} m_{\ell,k} u \,d\Omega.
\end{equation}
For convenience, we collect the basis elements of degree $\ell$ into vectors $\M[\ell] = (m_{\ell,-\ell},\dots,m_{\ell,\ell})^\Tr$, and we denote by $\bu_\ell = (u_{\ell,-\ell}, \dots, u_{\ell,\ell})^\Tr$ the vector of corresponding moments. Let $\mathbb{P}_N(\sphe)\subset L^2(\sphe)$ to be the span of all spherical harmonics with degree at most $N$.  Then the orthogonal projections $\cP_N \colon L^2(\sphe) \rightarrow \mathbb{P}_N(\sphe)$ and $\widetilde{\cP}_N \colon L^2(\sphe) \rightarrow \mathbb{P}_N(\sphe)$ are given by
\begin{equation}
\label{eq:projection}
\cP_N u = \sum_{\ell = 0}^N \M[\ell]^\Tr \bu_\ell=  \sum_{\ell = 0}^N \sum_{k = -\ell}^{\ell} m_{\ell,k}
u_{\ell,k}
\quand
\widetilde{\cP}_Nu=(\cI-\cP_N)u = \sum_{\ell = N+1}^\infty \M[\ell]^\Tr \bu_\ell=  \sum_{\ell = N+1}^\infty \sum_{k = -\ell}^{\ell} m_{\ell,k},
\end{equation}
where $\cI$ is the identity operator.

The $\Peqn$ approximation of \eqref{eq:transp_simple} seeks a function
\begin{equation}
 \psi^N(x,\Omega,t) = \sum_{\ell=0}^{N} \bmm_\ell^\Tr(\Omega) {\bm \uppsi}^N_\ell(x,t) = \sum_{\ell=0}^{N} \sum_{k=-\ell}^{\ell} m_{\ell,k}(\Omega) {\bm \uppsi}^N_{\ell,k}(x,t)
\end{equation}
 such that 
 \begin{subequations}
\begin{align}\label{eqn:pn_approx}
&    \e\partial_t\psi^{N}+\cP_N (\Omega\cdot\nabla_x\psi^{N})+\frac{\sigma}{\e}\psi^{N}=\frac{\sigma}{\e}\overline{\psi^{N}}+\e\cP_N q, \qquad \overline{\psi^{N}}=\frac{1}{4\pi}\int_{\sphe}\psi^N\,d\Omega,\\
& \psi^N|_{t=0} = \cP_N g.
 \end{align}
  \end{subequations}
When expressed in terms of the moments ${\bm \uppsi}^N_\ell$, the $\Peqn$ method yields the following linear, symmetric hyperbolic system:
\begin{subequations}
\label{eqn:pn_system}
\begin{flalign}
\e\partial_t \UN[0] & \,  &&+ \sum_{i=1}^3 \ali{1}{i}\partial_{x_i}\UN[1] &   &= \e\q_0, &&\text{for }\ell=0,
\\
\e\partial_t \UN[\ell]  &+ \sum_{i=1}^3  \alit{\ell}{i} \partial_{x_i}\UN[\ell-1] &&+ \sum_{i=1}^3  \ali{\ell+1}{i} \partial_{x_i}\UN[\ell+1]  &+ \frac{\sigma}{\e}\UN[\ell]&= \e\q_\ell, &&\text{for }1\leq \ell \leq N-1,
\\
\label{eqn:pn_system_N}
\e\partial_t \UN[N] &+ \sum_{i=1}^3 \alit{N}{i} \partial_{x_i}\UN[N-1] && \,  &+ \frac{\sigma}{\e}\UN[N]&= \e\q_N, &&\text{for } \ell = N.
\end{flalign}
\end{subequations}
Formulas for the elements in the matrices $\ali{\ell}{i} \in \mathbb{R}^{(2\ell-1)\times(2\ell+1)}$ can be found in the appendix of \cite{frank2016convergence}.  In the current work, we rely only on the fact the they are bounded in the operator norm, specifically that $\|\ali{\ell}{i}\|_2 \leq 4$.  

The exact moments $\U[\ell] = \int_{\bbS^2}\bmm_\ell \psi d \Omega$  satisfy an infinite system of equations with a structure similar to \eqref{eqn:pn_system}:
\begin{flalign}\label{eqn:pn_system_infinite}
\e\partial_t \U[0] & \,  &&+ \sum_{i=1}^3 \ali{1}{i}\partial_{x_i}\U[1] &   &= \e\q_0, &&\text{for }\ell=0,
\\
\e\partial_t \U[\ell]  &+ \sum_{i=1}^3  \alit{\ell}{i} \partial_{x_i}\U[\ell-1] &&+ \sum_{i=1}^3  \ali{\ell+1}{i} \partial_{x_i}\U[\ell+1]  &+ \frac{\sigma}{\e}\U[\ell]&= \e\q_\ell, &&\text{for } \ell \geq 1.
\end{flalign}
In particular, the $\Peqn$ equations for $\UN$ can be obtained by truncating \eqref{eqn:pn_system_infinite} at $\ell= N$ and then neglecting the moment $\UN[N+1]$ that would otherwise appear in \eqref{eqn:pn_system_N}.


\subsection{The hybrid method }
The hybrid method is based on a separation of $\psi$ into a \textit{collided component} $\psic$ and an \textit{uncollided component}  $\psiu$.  These components satisfy the coupled system
\begin{subequations}
\label{eq:hybrid}
\begin{align}
\e\partial_t \psiu + \Omega \cdot \nabla_x \psiu + \frac{\sigma}{\e}\psiu &= \e q, \\
 \e\partial_t \psic + \Omega \cdot \nabla_x \psic + \frac{\sigma}{\e}\psic &= \frac{\sigma}{\e} \psicbar + \frac{\sigma}{\e} \psiubar,
\end{align} 
\end{subequations}
where, as before, a bar denotes the angular average of $\sphe$.
The idea of the hybrid is to solve \eqref{eq:hybrid} using a high-resolution angular discretization for $\psiu$ and a low-resolution angular discretization for $\psic$ over a time step $\dt=T/M$, where $M\in\mathbb{N}_{>0}$ and then perform a reconstruction to reinitialize  $\psiu$ and $\psic$ for the next time step.  To formalize this procedure, define a set of temporal grid points $0=t_0 < t_1 < ...<t_{M}=T$, and for $m \in \{1,2,\dots,M\}$, let $f(t_{m}^-) = \lim_{\delta \to 0^+} f(t_{m} - \delta)$ for any function $f$ of $t$ that is continuous on $[t_{m-1},t_{m})$.  Then for $m \in \{1,2,\dots,M\}$, $\psi_{{\mathrm{u}},m}$ and $\psi_{{\mathrm{c}},m}$ satisfy the following system of equations over the interval $[t_{m-1},t_{m})$
\begin{subequations}
  \label{eq:hybrid_intervals}
  \begin{align}
    &\e\partial_t\psi_{{\mathrm{u}},m}+\Omega\cdot\nabla_x\psi_{{\mathrm{u}},m}+\frac{\sigma}{\e}\psi_{{\mathrm{u}},m}=\e q, \label{eq:hybrid_intervals_u} \\
    &\e\partial_t\psi_{\mathrm{c},m}+\Omega\cdot\nabla_x\psi_{\mathrm{c},m}+\frac{\sigma}{\e}\psi_{\mathrm{c},m}=\frac{\sigma}{\e}(\overline{\psi_{\mathrm{u},m}}+\overline{\psi_{\mathrm{c},m}}),
    \label{eq:hybrid_intervals_c}\\
    &\left.\psi_{\mathrm{u},m}\right|_{t=t_{m-1}}=
   \begin{cases}
   g, &\quad m = 1, 
   \\
   \psi_{\mathrm{u},m-1}(t_{m-1}^{-}) + \psi_{\mathrm{c},m-1}(t_{m-1}^{-}) &\quad m > 1.
   \end{cases}   \label{eq:relabel_u} \\
    &\left.\psi_{\mathrm{c},m}\right|_{t=t_{m-1}}=0.  \label{eq:relabel_c}
    \end{align}
\end{subequations}
The intuition behind this splitting is that \eqref{eq:hybrid_intervals_u} can be discretized with a high-resolution angular discretization but solved more efficiently than \eqref{eq:transp_simple} since angular unknowns are no longer coupled.  Although \eqref{eq:hybrid_intervals_c} features the same type of angular coupling as \eqref{eq:transp_simple}, it can be solved with fewer degrees of freedom because the source $\overline{\psi_{\mathrm{u},m}}$ is, in general, more regular than $q$. However, because $\psi_{{\mathrm{u}},m}$ decays exponentially whenever $\sigma > 0$, the hybrid is only solved for a time step $\dt$ before the relabeling in \eqref{eq:relabel_c}- \eqref{eq:relabel_u} is implemented.  The enables the hybrid to capture more high-resolution features than \eqref{eq:hybrid_intervals_c} can do alone.
%


\subsection{Angular discretization of the hybrid}

We focus now on the angular discretization of \eqref{eq:hybrid_intervals}.  The strategy of the hybrid is to discretize  \eqref{eq:hybrid_intervals_u} in angle with a high-resolution method and \eqref{eq:hybrid_intervals_c} in angle with a low-resolution method.  In practice, there are a variety of strategies and combinations available to do so.  For the purposes of analysis, we assume that \eqref{eq:hybrid_intervals_u} is solved exactly and that \eqref{eq:hybrid_intervals_c} is discretized with a $\Peqn$ method.  That is, we seek $\psi^{N} = \psi_{\mathrm{u},m}^{N} + \psi_{\mathrm{c},m}^{N}$ where for each $m \in \{1,2,\dots M\}$
\begin{equation}
    \left(\psi_{\mathrm{u},m}^{N}, \psi_{\mathrm{c},m}^{N} \right) \in C([t_{m-1},t_{m}); X \times L^2(\bbS^2)) \times C([t_{m-1},t_{m});X \times \bbP_{N}(\bbS^2)) 
\end{equation} 
satisfies%
\begin{subequations}
  \begin{align}
    &\e\partial_t\psi_{{\mathrm{u}},m}^{N}+\Omega\cdot\nabla_x\psi_{{\mathrm{u}},m}^{N}+\frac{\sigma}{\e}\psi_{{\mathrm{u}},m}^{N}=\e q,
      \label{eq:hybrid_intervals_PN_uncollided}\\
    &\e\partial_t\psi_{\mathrm{c},m}^{N}+\cP_N\left(\Omega\cdot\nabla_x\psi_{\mathrm{c},m}^{N}\right)+\frac{\sigma}{\e}\psi_{\mathrm{c},m}^{N}=\frac{\sigma}{\e}(\overline{\psi_{\mathrm{u},m}^{N}}+\overline{\psi_{\mathrm{c},m}^{N}}),
     \label{eq:hybrid_intervals_PN_collided}\\\\
    &\left.\psi_{\mathrm{c},m}^{N}\right|_{t=t_{m-1}}=0, \quad \left.\psi_{\mathrm{u},m}^{N}\right|_{t=t_{m-1}}=
   \begin{cases}
   g, &\quad m = 1,\\
   \psi_{\mathrm{u},m-1}^{N}(t^-_{m-1}) + \psi_{\mathrm{c},m-1}^{N}(t^-_{m-1}) &\quad m > 1.
   \end{cases}
    \label{eq:hybrid_intervals_PN_remap}
    \end{align}
        \label{eq:hybrid_intervals_PN}%
    \end{subequations}%
We compare the accuracy of the solution defined in \eqref{eq:hybrid_intervals_PN} with the monolithic $\Peqn$ method \eqref{eqn:pn_approx}, using the same value of $N$.  To simplify the numerical analysis of these two models, we keep the time and space variables continuous.

Since \eqref{eqn:pn_approx} and \eqref{eq:hybrid_intervals_PN_collided} have the same computational complexity, the goal is to asses the additional benefit of solving \eqref{eq:hybrid_intervals_PN_uncollided}. Clearly the additional cost of solving  \eqref{eq:hybrid_intervals_PN_uncollided} involves both memory and run-time; both are fairly easy to quantify. However, assessing the gains in accuracy is not as simple.  Thus is important to understand these gains in order to better quantify observed improvements in run-time efficiency provided by the hybrid.

\subsection{Preview of main results.}
Let $s\geq 1$ be the number of angular $L^2$ derivatives in the solution $\psi$ and let $N$ be an integer such that $N \geq s-1$. Let $e^N=\psi-\psi^N$ be the error in the $\Peqn$ approximation \eqref{eqn:pn_approx}, and let $e^N_M=\psi-(\psium[M]^N+\psicm[M]^N)$ be the error in the hybrid at the $M$-th time step.  The main results of the paper are the $\Peqn$ error estimate in Theorem \ref{thm:PN_error_est} and the hybrid error estimate in Theorem \ref{thm:hyb_final_err_est}. We compare these errors for two different regimes:  first, when $\sigma \asymp 1$%
\footnote{Recall that $a \asymp  b$ if and only if $a = O(b)$ and $b = O(a)$} and $\e \to 0$ (the diffusion regime) and second, when $\e \asymp (1)$ and $\sigma \to 0$ (the purely absorbing regime).%
\footnote{Technically, this is the streaming regime for $\Psi$, since there is no absorption.}

When $\e \ll 1$ and $\sigma \asymp 1$, Theorem \ref{thm:PN_error_est} implies that
    \begin{equation}
    \norma{e^N}_{L^2(X\times\sphe)}(T)\lesssim \frac{T}{(N+1)^s}\left[e^{-\sigma T/\e^2}\sum_{i=0}^{s-1}\frac{T^i}{\e^{i+1}}+\frac{\e^{s-1}}{\sigma^s}+O\left(\frac{\e}{\sigma}\right)\right].\label{eqn:pn_eps_small}
    \end{equation}
Thus with sufficient regularity, the $\Peqn$ approximation is spectrally accurate and grows linearly in time for large $T$.  The first and third term in brackets depend on whether or not $q$ and $g$ are isotropic (independent of angle).   The first term is due to initial layers when $g$ is non-isotropic, and the third terms arises when $q$ is non-isotropic.  When $g$ and $q$ are isotropic, the $\Peqn$ error reduces to (see Corollary \ref{corl:PN_error_est_isotropic})
    \begin{equation}
            \norma{e^N}_{L^2(X\times\sphe)}(T)\lesssim \frac{\e^{s-1}T}{\sigma^s(N+1)^s}.
    \end{equation}
    
For the hybrid method the initial condition and source are always isotropic.  Thus Theorem \ref{thm:hyb_final_err_est} gives the following compact bound
   \begin{equation}
            \norma{e_M^N}_{L^2(X\times\sphe)}(T)\lesssim \frac{\e^{s-1}T}{\sigma^s(N+1)^s}. \label{eqn:hyb_eps_small}
    \end{equation}
Thus the hybrid method comes equipped with a better error estimate than the $\Peqn$ method, and the estimates agree when the data is isotropic.  This suggests that the hybrid method is at least as accurate as the $\Peqn$ approximation when $\e \ll 1$ and $\sigma \asymp 1$.  In addition, the hybrid estimate is independent of the time step $\dt$ used for re-initialization in this regime.   

For both the hybrid method and $\Peqn$ approximation, the errors converge to zero as $\e \to 0$ whenever $s>1$ (modulo the initial layer in \eqref{eqn:pn_eps_small}). This fact is consistent with the fact that the $\Peqn$ method recovers the diffusion limit \eqref{eq:diff_limit} whenever $N \geq 1$ \cite{hauck2009temporal}.

For the second regime of interest, $\sigma \ll 1$ and $\e \asymp 1$, Theorem \ref{thm:PN_error_est} gives the $\Peqn$ error estimate   
    \begin{equation}
    \norma{e^N}_{L^2(X\times\sphe)}(T)\lesssim \frac{\mathfrak{p}_s(T/\e)}{(N+1)^s},
    \end{equation}
    where $\mathfrak{p}_{s}(\omega)=\sum_{i=0}^{s+1}c_i(T)\omega^i$ is a polynomial of degree $s+1$ with non-negative coefficients $c_i(T)=a_i T+b_i$, $a_i,\,b_i\geq 0$.  Here the hybrid estimate provides a significant improvement:
\begin{equation}     \norma{e_M^N}_{L^2(X\times\sphe)}(T)\lesssim \frac{\Delta t^{s}T}{\e^{s+1}(N+1)^s}\min(1,\frac{\Delta t\sigma}{\e^2})
\end{equation}
In particular, $\norma{e_M^N}_{L^2(X\times\sphe)}(T)=0$, when $\sigma=0$.  This result is expected since in that case the uncollided solution and the transport solution agree.  As expected, the error is monotonic in $\dt$; however, small time steps require more evaluations of the uncollided equation and more reintializions.  In practice, this additional cost must be taken into account.

\section{ \texorpdfstring{$\Peqn$}{PN} Analysis}
\label{sec:pn-analysis}
\subsection{Spherical harmonics preliminaries}
A natural space to analyze the transport equation and the $\Peqn$ approximation is the Sobolev space $H_{\circ}^s(\sphe)$.  To describe this space, we recall some elementary facts about spherical harmonics which can be found, for example, in \cite{atkinson2012spherical,dai2013}. For $u,v \in L^{2}(\sphe)$, let
\begin{equation}
    (u,v)_{L^2(\sphe)} = \sum_{\ell=0}^{\infty} \bu_\ell ^\top \mathbf{v}_{\ell}
    \qquand
    \|u\|_{L^2(\sphe)}^2 =  \sum_{\ell=0}^{\infty} \|\bu_\ell\|^2,
\end{equation}
where
\begin{equation}
    \bu_\ell ^\top \mathbf{v}_{\ell} = \sum_{k = -\ell}^{\ell} u_{\ell,k} v_{\ell,k} 
    \qquand
    \|\bu_\ell\|^2 
    =  \bu_\ell ^ \top \bu_\ell
    = \sum_{k=-\ell}^{\ell}|u_{\ell,k}|^2.
\end{equation}

A standard way to define Sobolev spaces on the sphere is via the Laplace-Beltrami operator $\Delta_{\circ}$, which is the spherical component of the Laplacian and for which the spherical harmonics are eigenfunctions:
   $ -\Delta_{\circ} Y_{\ell,j} = \ell(\ell+1)Y_{\ell,j}.$ For even integers $s$, the usual norm is
\begin{equation}
\label{eq:H0s}
 \|u\|_{H^s_{\circ}(\sphe)} :=  \left\|\left(\frac{1}{4}+\Delta_{\circ}\right)^{s/2} u\right\| = \left(\sum_{\ell=0}^{\infty} b_{\ell}  \|\bu_\ell \|^2\right)^{1/2},
 \qquad 
    b_{\ell} = \left(\frac{1}{2}+\ell \right)^{2s}.
\end{equation}
The definition of this inner product extends naturally to all $s \in \bbR$, and the space $H^s_{\circ}(\sphe)$ is then the completion of smooth functions under the $H^s_{\circ}(\sphe)$ norm \cite{atkinson2012spherical}.

Rather than working directly with the $H^s_{\circ}(\sphe)$ norm in \eqref{eq:H0s}, it is convenient in the analysis below to use an equivalent norm. For $s\geq 0$, define the $H^{s}(\sphe)$ semi-norm  and norm by
\begin{gather}
\label{eq:Hs}
|u|_{H^s(\sphe)} := \left(\sum_{\ell=s}^{\infty}b_{\ell} \|\bu_\ell \|^2\right)^{1/2}
\quand
\|u\|_{H^s(\sphe)} = \left(s\|u\|_{L^2(\sphe)}^2 + |u|^2_{H^s(\sphe)} \right)^{1/2},
\end{gather}
respectively, where the sum in the semi-norm definition in \eqref{eq:Hs} begins at $s$ for technical arguments that are used in the proof of Lemma \eqref{lem:semi_norm_rec} below.   
When $s=0$, the norms coincide:  $\|u\|_{H^0(\sphe)} = \|u\|_{H^0_{\circ}(\sphe)} = \|u\|_{L^2(\sphe)}$.  More generally, the following equivalence holds.
\begin{lem}[Norm equivalence]
For any $s \geq 0$,
\begin{equation}
    c_1(s)\|u\|_{H^s(\sphe)} \leq \|u\|_{H^s_{\circ}(\sphe)} \leq c_2(s)\|u\|_{H^s(\sphe)}.\label{eq:norm_equiv}
\end{equation}
where 
\begin{equation}
\label{eq:equiv_const}
    c_1(s) =\begin{cases}
     1&\text{if}\quad s=0,\\
    \frac{1}{\sqrt{3s}}&\text{if}\quad s\geq 1
    \end{cases}
    \quand 
    c_2(s) =\begin{cases}
    1&\text{if}\quad s=0,\\
    \sqrt{\frac{5}{s}}\left(s-\frac{1}{2}\right)^s&\text{if}\quad s\geq 1
    \end{cases}.    
\end{equation}
\end{lem}
\begin{proof} When $s=0$, the norms are equal, so \eqref{eq:norm_equiv} holds trivially.   Thus assume that $s \geq 1$.  The first inequality in \eqref{eq:norm_equiv} follows from the fact that  
\begin{equation}
\begin{split}
\|u\|_{H^s(\sphe)}^2
                    &= s\sum_{\ell=0}^{\infty}\|\bu_\ell \|^2+\sum_{\ell=s}^{\infty}\halfl^{2s}\|\bu_\ell \|^2
                    \leq 3s\sum_{\ell=0}^{\infty}\left(\frac{1}{2}+\ell\right)^{2s}\|\bu_\ell \|^2 = \frac{1}{[c_1(s)]^2}\|u\|_{H^s_{\circ}(\sphe)}^2.
\end{split}
\end{equation}
To prove the second inequality in \eqref{eq:norm_equiv}, we use the elementary inequality
\begin{equation}
    \frac{s}{4} \leq \left(s-\frac{1}{2}\right)^{2s}, \quad  s\geq 1,
\end{equation}
to conclude that
\begin{equation}
\begin{split}
    \|u\|_{H^s_{\circ}(\sphe)}^2
    &=\sum_{\ell=0}^{s-1}\left(\frac{1}{2}+\ell\right)^{2s}\|\bu_\ell \|^2
        +|u|_{H^s(\sphe)}^2
    \leq \frac{1}{s}\left(s-\frac{1}{2}\right)^{2s}s\norma{u}_{L^2(\sphe)}^2+|u|_{H^s(\sphe)}^2\\
    &\leq \frac{1}{s}\left(s-\frac{1}{2}\right)^{2s}s\norma{u}_{L^2(\sphe)}^2+
     \frac{4}{s}\left(s-\frac{1}{2}\right)^{2s} |u|_{H^s(\sphe)}^2 \\
    &\leq\frac{5}{s}\left(s-\frac{1}{2}\right)^{2s}\left(s\norma{u}_{L^2(\sphe)}^2+|u|_{H^s(\sphe)}^2\right) 
    = [c_2(s)]^2\|u\|_{H^s(\sphe)}^2. \qquad \qquad \qquad 
\end{split}
\end{equation}
\end{proof}

\begin{lem}[Approximation property]
\label{lem:approx}
For  $s\geq 0$ and $N \geq \max\{0,s-1\}$,
\begin{equation}
 \|(\cI-\cP_N)u\|_{L^2(\sphe)} \leq \frac{1}{(N+1)^s} |(\cI-\cP_N)u|_{H^s(\mathbb{S}^2)} \leq \frac{1}{(N+1)^s} |u|_{H^s(\mathbb{S}^2)}.
\end{equation}
\end{lem}\begin{proof}
From the definition of the projection $\cP_N$ in \eqref{eq:projection},
\begin{align}
\underbrace{\sum_{\ell=N+1}^{\infty}\|\bu_\ell \|^2}_{=\|(\cI-\cP_N)u\|^2_{L^2(\sphe)} }
\leq \frac{1}{(N+1)^{2s}} \underbrace{\sum_{\ell= N+1}^{\infty}\halfl^{2s}\|\bu_\ell \|^2}_{=|(\cI-\cP_N)u|^2_{H^s(\mathbb{S}^2)}}
\leq \frac{1}{(N+1)^{2s}} \underbrace{\sum_{\ell= s}^{\infty}\halfl^{2s}\|\bu_\ell \|^2}_{=|u|^2_{H^s(\mathbb{S}^2)}}
\end{align}
Taking square roots of each term above yields the desired result.
\end{proof}
For vector-valued functions of space, we define the usual $L^2(X)$ inner-product and norm by
\begin{equation}
    (\mathbf{v},\mathbf{w})_{L^2(X)}=\int_X  \mathbf{v}(x)^{\top}\mathbf{w}(x)\,dx 
    \quand
    \norma{\mathbf{v}}^2_{L^2(X)}
    =(\mathbf{v},\mathbf{v})_{L^2(X)}
    =\int_X \|\mathbf{v}(x)\| ^2 \,dx. 
\end{equation}
%
The  space $H^{0,s}(X \times \sphe) = L^2(X;H^s(\sphe))$ is the space of measurable functions $ u\colon X \times \sphe \to \bbR$ with the semi-inner product 
\begin{gather}
\label{eq:H-zero-s_prod}
 ( u, v)_{H^{0,s}(X \times \sphe)}
 = \int_X\sum_{\ell=s}^{\infty} \halfl^{2s}\bu_\ell(x) ^\top \mathbf{v}_{\ell}(x)\,dx 
 = \sum_{\ell=s}^{\infty} \halfl^{2s}(\bu_\ell,\mathbf{v}_\ell)_{L^2(X)}
\end{gather}
such that the semi-norm
\begin{gather}
\label{eq:H-zero-s_norm}
 |u|^2_{H^{0,s}(X \times \sphe)}
 = \int_X\sum_{\ell=s}^{\infty} \halfl^{2s}|\bu_\ell(x)|^2 dx 
 = \sum_{\ell=s}^{\infty} \halfl^{2s}\|\bu_\ell\|_{L^2(X)}^2
\end{gather}
is bounded.

For $r\in\mathbb{N}_{\geq 0}$ and $s \geq 0$ we define $H^{r,s}(X \times \sphe)$ to be the space of functions $u \colon X \times \sphe \to \bbR$ such that the semi-norms. 
\begin{gather}
\label{eq:H-r-s_norm}
 |u|_{H^{\varsigma,s}(X \times \sphe)}:= \sum_{i_1,i_2,\dots,i_{\varsigma}=1}^3|\partial_{x_{i_1}x_{i_2}\dots x_{i_{\varsigma}}} u|_{H^{0,s}(X \times \sphe)}.
\end{gather}
are bounded for all positive integers $\varsigma \leq r$. These semi-norms above are equivalent to the standard semi-norms, but are more convenient in the context of the RTE since
\begin{equation}
 |\Omega \cdot \nabla_x u|_{H^{r,s}(X \times \sphe)}  \leq \sum_{i=1}^3|\partial_{x_i}u|_{H^{r,s}(X \times \sphe)} = |u|_{H^{r+1,s}(X \times \sphe)},
\end{equation}
which will be used in the analysis below. Henceforth, the domain of integration for $H^{s}$ and $H^{r,s}$ will be left off when 
there is no ambiguity, i.e.,
\begin{equation}
    |u|_{H^{s}}:=|u|_{H^{s}(\sphe)} \quand
    |u|_{H^{r,s}}:=|u|_{H^{r,s}(X \times \sphe)}
\end{equation}
Finally, for $p \geq 1$ and measurable functions $u \colon X \times \bbS^2 \times [\alpha,\beta] \to \bbR$, we denote the space-time semi-norms by
\begin{equation}
|u|_{L^p([\alpha,\beta];H^{r,s})} = \left(\int_{\alpha}^\beta |u|_{H^{r,s}}^p\,d\tau\right)^{1/p}
\quand
|u|_{L^\infty([\alpha,\beta]);H^{r,s})} = \operatorname*{ess~sup}_{t \in [\alpha,\beta]} {|u|_{H^{r,s}}}.
\end{equation}
\subsection{Stability of the \texorpdfstring{$\Peqn$}{PN} system}
In this section, we derive estimates on high-order semi-norms that arise in the subsequent error analysis. The analysis requires iterated inequalities of Gr\"onwall type (see Lemma \ref{lem:basic_ode_result} in the Appendix), so we define the following notation to simplify integrals that arise.  
Let $0\leq \alpha \leq t \leq \beta < \infty$, and for any function $f \in L^1([\alpha,\beta])$, define the bounded linear operator $\cA_{\alpha} \colon L^1([\alpha,\beta]) \rightarrow C^0([\alpha,\beta])$ and its powers $\cA^k_{\alpha}, k \in \bbN^{\geq0}$, by
\begin{equation}\label{eq:Ak_def}
    \cA_{\alpha}[f] (t) := \frac{1}{\e}\int_{\alpha}^t e^{-\sigma(t- \tau)/\e^2} f(\tau)  d\tau
    \quand
    \cA^k_{\alpha}[f](t) =  \begin{cases}
    \cI , & k =0, \\
    \cA_{\alpha}[\cA^{k-1}_{\alpha}[f]](t), & k \geq 1.
    \end{cases} 
\end{equation}
In addition, let $\mathbbm{1}\in L^1([\alpha,\beta])$ be the function that is identically one, and let
\begin{equation}
\label{eq:test_f}
    F_{\alpha}(t) = e^{-\sigma(t-\alpha)/\e^2},  \quad t \in [\alpha,\beta].
\end{equation}
It is clear from the definition in \eqref{eq:Ak_def} $\cA_{\alpha}$ is monotonic; that is, if $0\leq f(t)\leq g(t)$ for a.e. $t\in[\alpha,\beta]$, then \mbox{$\cA_{\alpha}[f](t)\leq \cA_{\alpha}[g](t)$} for all $t\in[\alpha,\beta]$.
\begin{lem}\label{lem:A_estimates} 
Let $F_{\alpha}$ be given as in \eqref{eq:test_f}. Then for all $t \geq \alpha$ and every $k \in \bbN_{\geq0}$,
\begin{equation}\label{eq:A_one}
\cA^k_{\alpha}[ \mathbbm{1}](t)\leq \min\left(\frac{\e^k}{\sigma^k}, \frac{1}{k!}\left(\frac{t-\alpha}{\e}\right)^k\right)
\quand
\cA^k_{\alpha}[F_{\alpha}](t)
=\frac{(t-\alpha)^k e^{-\sigma(t-\alpha)/\e^2}}{ k! \e^k}
\end{equation}
 \end{lem}
\begin{proof}
We first prove the bound in \eqref{eq:A_one}. Since $e^{-\sigma(t-\tau)/\e^2} \leq 1$, a direct calculation gives
\begin{equation}\label{eq:Ak_formula}
\cA^k_{\alpha}[\mathbbm{1}](t) \leq \frac{1}{\e^k}
\int_{\alpha}^t \int_{\alpha}^{\tau_{k-1}} \cdots \int_{\alpha}^{\tau_{1}} \mathbbm{1} d\tau_0\cdots d\tau_{k-1} 
= \frac{1}{k!}\left(\frac{t-\alpha}{\e}\right)^k.
\end{equation}
On the other hand, it follows directly from the definition of $\cA_{\alpha}$ that 
\begin{equation}\label{eq:Ak_formula_2}
    \cA_{\alpha}[\mathbbm{1} ](t) = \frac{\e}{\sigma}\left( 1 -F_{\alpha}(t)\right)\leq \frac{\e}{\sigma} \quad \implies \quad \cA^k_{\alpha}[\mathbbm{1} ](t)\leq \frac{\e^k}{\sigma^k} 
\end{equation}
Together \eqref{eq:Ak_formula} and \eqref{eq:Ak_formula_2} yield \eqref{eq:A_one}.

We prove the second statement in \eqref{eq:A_one} by induction on $k$. When $k=0$, the statement follows trivially,
\begin{equation}
    \cA_{\alpha}^0[F_{\alpha}](t)=F_{\alpha}(t)=\frac{(t-\alpha)^0e^{-\sigma(t-\alpha)/\e^2}}{0!e^{0}},
\end{equation}
Now let us assume the statement is true for arbitrary $k$, and 
\begin{equation}
\cA_{\alpha}^{k+1}[F_{\alpha}](t)=  \cA_{\alpha}\left[\cA_{\alpha}^{k}[F_{\alpha}]\right](t)=\frac{1}{\e}\int_{\alpha}^t e^{-\sigma(t-\tau)/\e^2}\frac{(\tau-\alpha)^k e^{-\sigma(\tau-\alpha)/\e^2}}{k!\e^k}\,d\tau=\frac{(t-\alpha)^{k+1}e^{-\sigma(t-\alpha)/\e^2}}{(k+1)!\e^{k+1}}
\end{equation}
\end{proof}
\subsubsection{Estimates for the \texorpdfstring{$\Peqn$}{PN} system}
We derive evolution equations for the $H^{r,s}$ semi-norms, defined in \eqref{eq:H-r-s_norm}. For $r=0$, we test each equation in \eqref{eqn:pn_system} by $b_{\ell}\UN[\ell]$, integrate by parts, and then sum over $\ell \in \{s,s+1, ..., N \}$.  This gives
\begin{equation}
\begin{split}
\label{eqn:weighted_energy_equality}
\frac{\e}{2} \partial_t  |\psi^N|_{H^{0,s}}^2 
&+ \frac{\sigma}{\e}|\psi^N|_{H^{0,s}}^2 
= \e  ( \cP_N q, \psi^N)_{H^{0,s}} \\
& \quad - \sum_{i=1}^3 \sum_{\ell=s}^{N} \left(\left(\ell+\frac12\right)^{2s} 
-  \gamma_{s,\ell}  \left(\ell-\frac12\right)^{2s} \right) \left(\UN[\ell], \left(\ali{\ell}{i}\right)^T \partial_{x_i}\UN[\ell-1]\right)_{L^2(X)},
\end{split}    
\end{equation}
where $\gamma_{s,\ell} = (1 - \delta_{s,\ell})$ is used to handle the first non-zero term in the sum over $\ell$.
For the special case $s=0$, \eqref{eqn:weighted_energy_equality} recovers the usual $L^2$ energy equation:
\begin{equation}
\label{eq: PN-L2}
    \frac{\e}{2}\frac{d}{dt}\norma{\psi^N}_{L^2(X \times \sphe)}^2 + \frac{\sigma}{\e}\norma{\psi^N - \overline{\psi^N }}_{L^2(X \times \sphe)}^2
    =\e(\cP_N q,\psi^N)_{L^2(X \times \sphe)}^2.
\end{equation}
To find a closed estimate with respect to the $H^{0,s}$ semi-norms, we focus on the summation in \eqref{eqn:weighted_energy_equality}. 
\begin{lem}
\label{lem:s-l_bounds}
Let $s\geq 1$ and $\ell \geq s$.
Then 
\begin{equation}
\label{eq:s-l_bounds}
\left(\ell+\frac12\right)^{2s} 
-  \gamma_{s,\ell}  \left(\ell-\frac12\right)^{2s} 
\leq 2es\left(\ell+\frac12\right)^s \left(\ell-\frac12\right)^{s-1}.
\end{equation}
\end{lem}
\begin{proof} We first establish an elementary inequality. Since $\ell \geq s$,
\begin{equation}
    \ell+\frac12
    = \left(\ell-\frac12\right) \left(1 + \frac{1}{\ell-\frac12}\right) 
    \leq \left(\ell-\frac12\right) \left(1 + \frac{1}{s-\frac12}\right) 
    \leq \left(\ell-\frac12\right)  e^{1/(s-1/2)}
\end{equation}
Therefore
\begin{equation}
\label{eq:shift_bound}
    \left(\ell+\frac12\right)^{s-1} 
    \leq e^{\frac{s-1}{s-1/2}} \left(\ell-\frac12\right)^{s-1}
    \leq e \left(\ell-\frac12\right)^{s-1} .
\end{equation}
We use \eqref{eq:shift_bound} to show \eqref{eq:s-l_bounds}.  There are two cases:
\begin{description}
\item[Case 1 ($\ell = s$):] In this case, $\gamma_{s,\ell} = 0$. Since $2s > s+1/2$ and by \eqref{eq:shift_bound},
\begin{equation}
\left(\ell+\frac12\right)^{2s} 
     = \left(s+ \frac12\right) \left(\ell+\frac12\right)^{s}  \left(\ell+\frac12\right)^{s-1} 
     \leq 2e s\left(\ell+\frac12\right)^{s}  \left(\ell-\frac12\right)^{s-1}
\end{equation}
\item[Case 2 ($\ell > s$):]In this case, $\gamma_{s,\ell} = 1$.  
Applying a binomial expansion and then \eqref{eq:shift_bound} gives
\begin{equation}
\begin{split}
&\left(\ell+\frac12\right)^{2s} - \left(\ell-\frac12\right)^{2s} 
 = \sum_{k=0}^{2s}\binom{2s}{k}\left(\ell-\frac12\right)^k - \left(\ell-\frac12\right)^{2s}
 = \sum_{k=0}^{2s-1}\binom{2s}{k}\left(\ell-\frac12\right)^k \\
 &\quad = \sum_{k=0}^{2s-1}\frac{2s}{2s-k}\binom{2s-1}{k}\left(\ell-\frac12\right)^k 
 \leq 2s\sum_{k=0}^{2s-1}\binom{2s-1}{k}\left(\ell-\frac12\right)^k\\
 &\quad = 2s\left(\ell+\frac12\right)^{2s-1}
 = 2s\left(\ell+\frac12\right)^{s}\left(\ell+\frac12\right)^{s-1}
 \leq 2es\left(\ell+\frac12\right)^{s}\left(\ell-\frac12\right)^{s-1}
\end{split}
\end{equation}
\end{description}

\end{proof}
\begin{lem}[Semi-norm recurrence]\label{lem:semi_norm_rec}
Let $s \geq 1$, $q \in L^1([0,T];H^{r,s})$ and  $g \in H^{r,s}$. Then for all $t\in [0,T]$,
\begin{gather}
\label{eq:semi-norm_recurrence}
|\psi^N|_{H^{r,s}}(t)\leq Cs \cA_{0}[|\psi^N|_{H^{r+1,s-1}}](t) +  |\cP_N g|_{H^{r,s}}F_0(t) +\e  \cA_{0}[|\cP_N q|_{H^{r,s}}](t),
\end{gather}
where $F_0$ is defined in \eqref{eq:test_f} and $C$ is a constant independent of the data.
\end{lem}
\begin{proof}We assume first that $r = 0$ and focus on the last term in \eqref{eqn:weighted_energy_equality}.  It follows from (i) the induced norm bound $\|a_\ell^{(i)}\|_{2} \leq 4$ \cite{frank2016convergence}, (ii) the bounds in Lemma \ref{lem:s-l_bounds}, (iii) the Cauchy-Schwarz inequality, and (iv) the $H^{r,s}$ semi-norm definitions in \eqref{eq:H-zero-s_norm} and \eqref{eq:H-r-s_norm} that 
\begin{equation}
\begin{split}
& \sum_{\ell=s}^{N} \left(\left(\ell+\frac12\right)^{2s} 
-  \gamma_{s,\ell}  \left(\ell-\frac12\right)^{2s} \right) \left(\UN[\ell], \left(\ali{\ell}{i}\right)^T \partial_{x_i}\UN[\ell-1]\right)_{L^2(X)} \\
& \quad    \leq 2es \sum_{\ell=s}^{N} \left(\ell+\frac12\right)^s \left(\ell-\frac12\right)^{s-1}
    \|\ali{\ell}{i}\|_2 \enspace
    \|\UN[\ell]\|_{L^2(X)} \enspace
    \| \partial_{x_i}\UN[\ell-1]\|_{L^2(X)}
    \\
& \quad    
\leq Cs 
\left(\sum_{\ell=s}^{N} \left(\ell+\frac12\right)^{2s}  |\UN[\ell]|^2_{L^2(X)} \right)^{1/2}
\left(\sum_{\ell=s}^{N} \left(\ell-\frac12\right)^{2(s-1)} |
\partial_{x_i}\UN[\ell-1]|^2_{L^2(X)}\right)^{1/2} \\
& \quad \leq Cs |\psi^N|_{H^{0,s}}|\partial_{x_i}\psi^N|_{H^{0,s-1}}.
\end{split}
\end{equation}
Applying the bound above to the right-hand side of \eqref{eqn:weighted_energy_equality} and applying Lemma \ref{lem:basic_ode_result} gives  \eqref{eq:semi-norm_recurrence}.
The case $r > 1$ can be handled by differentiating the $\Peqn$ equations in space and then repeating the arguments above.
\end{proof}

For a general function $\phi\in L^p ([\alpha,\beta];H^{r,s})$, let \mbox{$|\phi|_{L^p([\alpha,\bullet];H^{0,r+s})}: [\alpha,\beta] \rightarrow \mathbb{R} $} be the map defined
\begin{equation*}
    |\phi|_{L^p([\alpha,\bullet];H^{r,s})}(\tau) = |\phi|_{L^p([\alpha,\tau];H^{r,s})}.
\end{equation*}

\begin{lem}[Stability of higher-order semi-norms]
\label{lem:stability_higher_order} Let $q \in L^1([0,T];H^{r,0})$ and $g \in H^{r,0}$.  If $t\in [0,T]$, then
\begin{gather}\label{eqn:PN_l2_stability}
    |\psi^N|_{H^{r,0}}(t)\leq |\cP_N g|_{H^{r,0}}+|\cP_N q|_{L^1([0,t];H^{r,0})}.
\end{gather}
If, in addition, $s\geq 1$, $q \in L^1([0,T];H^{i,j})$ and $g \in H^{i,j}$ for each $i,j$ such that $0 \leq i \leq r$, $0 \leq j \leq s$,  and $i+ j = r+s$,
\begin{equation}\label{eqn:hrs_stability}
\begin{split}
|\psi^N|_{H^{r,s}}(t)
   &
        \leq C_s s!  \cA_{0}^s\left[|\cP_N g|_{H^{r+s,0}}+|\cP_N q|_{L^1([0, \bullet];H^{r+s,0})} \right](t)\\
    &+ C_s\sum_{i=0}^{s-1}\frac{s!}{(s-i)!} \cA_{0}^i[ F_0|\cP_N g|_{H^{r+i,s-i}}](t) 
        +C_s\e\sum_{i=0}^{s-1}\frac{s!}{(s-i)!} \cA_{0}^{i+1}[|\cP_N q|_{H^{r+i,s-i}}](t),
\end{split}
\end{equation}
where $F_0$ is given in \eqref{eq:test_f} where $C_s$ is a constant depending only on $s$.
\end{lem}

\begin{proof}
First we will prove \eqref{eqn:PN_l2_stability} for $s=0$, in which case $H^{r,0} = {L^2(X \times \sphe)}$. From \eqref{eq: PN-L2} and the Cauchy-Schwarz inequality,
\begin{equation}
    \frac{\e}{2}\frac{d}{dt}\norma{\psi^N}_{L^2(X \times \sphe)}^2
        \leq \e(\cP_N q,\psi^N)_{L^2(X \times \sphe)}
        \leq \e \norma{\cP_N q}_{L^2(X \times \sphe)}\norma{\psi^N}_{L^2(X \times \sphe)}
\end{equation}
an application of Lemma \ref{lem:basic_ode_result}, gives
\begin{gather}
\label{eq:PN_L2_bound}
    \norma{\psi^N}_{L^2(X \times \sphe)}(t)\leq \norma{\cP_N g}_{L^2(X \times \sphe)}+\norma{\cP_N q}_{L^1([0,t];L^2(X \times \sphe))}.
\end{gather}
For $r >0$, $\phi^N=\partial_{x_{i_1}x_{i_2}\dots x_{i_r}}\psi^N $ satisfies \eqref{eqn:pn_approx}, with initial condition $\left.\phi^N\right|_{t=0}=\cP_N(\partial_{x_{i_1}x_{i_2}\dots x_{i_r}}g) = \partial_{x_{i_1}x_{i_2}\dots x_{i_r}}\cP_N g$ and source $\cP_N(\partial_{x_{i_1}x_{i_2}\dots x_{i_r}}q) = \partial_{x_{i_1}x_{i_2}\dots x_{i_r}} \cP_N q$.  Repeating the argument above gives, in analogy with \eqref{eq:PN_L2_bound},
\begin{gather}
\label{eq:Hr}
    \norma{\partial_{x_{i_1}x_{i_2}\dots x_{i_r}}\psi^N}_{L^2(X \times \sphe)}(t)
    \leq \norma{\partial_{x_{i_1}x_{i_2}\dots x_{i_r}}\cP_N g}_{L^2(X \times \sphe)}
        +\norma{\partial_{x_{i_1}x_{i_2}\dots x_{i_r}}\cP_N q}_{L^1([0,t];L^2(X \times \sphe))},
\end{gather}
Summing \eqref{eq:Hr} over each permutation of the $r$ derivatives $\partial_{x_{i_1}}\partial_{x_{i_2}}\cdots\partial_{x_{i_r}}$ and using the definitions in \eqref{eq:H-r-s_norm} recovers \eqref{eqn:PN_l2_stability}.

We next prove \eqref{eqn:hrs_stability}.  Let
\begin{gather}
\label{eq:term_recursion_defs}
    b_{r,s}(t)=|\cP_N g|_{H^{r,s}}F_0(t) +\e  \cA_{0}[|\cP_N q|_{H^{r,s}}](t)
    \quand
    c_{r,s}(t)=|\psi^N|_{H^{r,s}}(t),
\end{gather}
where $F_0$ is defined in \eqref{eq:test_f}.
Then Lemma \ref{lem:semi_norm_rec} gives the following recursion relation:
\begin{equation}
\label{eq:recurrence}
c_{r,s}(t) \leq  Cs \cA_{0}[c_{r+1,s-1}](t) + b_{r,s}(t), \qquad  s \geq 1.
\end{equation}
Unrolling this recursion in $s$ gives
\begin{equation}    
    c_{r,s}(t)
        \leq s!C^s  \cA^s_{0}[ c_{r+s,0}](t) + \sum_{i=0}^{s-1}\frac{s!}{(s-i)!}C^i \cA_{0}^i [b_{r+i,s-i}](t),
\end{equation}
and translating back to the semi-norms with \eqref{eq:term_recursion_defs} gives
\begin{equation}
\begin{split}
    |\psi^N|_{H^{r,s}}(t) 
        &\leq s!C^s  \cA_{0}^s[|\psi^N|_{H^{r+s,0}}](t)
        \\&+ \sum_{i=0}^{s-1}\frac{s!}{(s-i)!}C^i |\cP_N g|_{H^{r+i,s-i}} \cA_{0}^i[F_0](t)
         +\e\sum_{i=0}^{s-1}\frac{s!}{(s-i)!}C^{i+1} \cA_{0}^{i+1}[|\cP_N q|_{H^{r+i,s-i}}](t) 
\label{eqn:stability_inequality_proof}.
\end{split}
\end{equation}
We then apply \eqref{eqn:PN_l2_stability} (with $r$ replaced with  $r + s$) to the first term on the right hand side of  \eqref{eqn:stability_inequality_proof}.

\end{proof}

 \subsubsection{Estimates for continuous system}
In this section we extend the stability results for $|\psi^N|_{H^{r,s}}$ to $|\psi|_{H^{r,s}}$, using the infinite moment hierarchy in \eqref{eqn:pn_system_infinite}.  These estimates will be useful in deriving consistency estimates.  For the case $r=0$, as in the previous section, we test each equation in \eqref{eqn:pn_system_infinite} by $b_{\ell}\U[\ell]$,  integrate by parts, and sum over $\ell\geq s$. The result is analogous to \eqref{eqn:weighted_energy_equality}, namely
\begin{equation}
\begin{split}
\label{eqn:weighted_energy_equality_cont}
\frac{\e}{2} \partial_t  |\psi|_{H^{0,s}}^2 
&+ \frac{\sigma}{\e}|\psi|_{H^{0,s}}^2 
= \e  ( q, \psi)_{H^{0,s}} \\
& \quad - \sum_{i=1}^3 \sum_{\ell=s}^{\infty} \left(\left(\ell+\frac12\right)^{2s} 
-  \gamma_{s,\ell}  \left(\ell-\frac12\right)^{2s} \right) \left(\U[\ell], \left(\ali{\ell}{i}\right)^T \partial_{x_i}\U[\ell-1]\right)_{L^2(X)},
\end{split}    
\end{equation}
As before, when $s=0$, \eqref{eqn:weighted_energy_equality_cont} recovers the usual $L^2$ stability result:
\begin{equation}
\label{eq: regular-L2}
    \frac{\e}{2}\frac{d}{dt}\norma{\psi}_{L^2(X \times \sphe)}^2 + \frac{\sigma}{\e}\norma{\psi - \overline{\psi}}_{L^2(X \times \sphe)}^2
    =\e(q,\psi)_{L^2(X \times \sphe)}^2
\end{equation}
The following results are continuous analogues to  Lemmas \ref{lem:semi_norm_rec} and \ref{lem:stability_higher_order}.  Their proofs are nearly identical, so we only give a brief summary.
 \begin{lem}[Semi-norm recurrence for the continuous system]\label{lem:semi_norm_rec_cont}
 Let $s\geq 1$, $q \in L^1([0,t];H^{r,s})$ and $g \in H^{r,s}$.  Then for all $t\in[0,T]$,
 \begin{gather}
 \label{eq:semi-norm_recurrence_cont}
|\psi|_{H^{r,s}}(t)\leq Cs \cA_{0}[|\psi|_{H^{r+1,s-1}}](t) +  |g|_{H^{r,s}}F_0(t) +\e  \cA_{0}[|q|_{H^{r,s}}](t),
\end{gather}
where $C$ is a constant independent of the data. 
 \end{lem}
 \begin{proof}[Summary of Proof] The proof follows the same lines and with same constants as in Lemma \ref{lem:semi_norm_rec} after changing $\UN[\ell]$ by $\U[\ell]$ and then taking all the sums to infinity. To generalize the proof for $r>1$, we differentiate the system  \eqref{eqn:pn_system_infinite} in space and repeat the process.\end{proof}

 \begin{lem}[Stability of higher order semi-norms for the continuous system]\label{lem:stability_higher_order_cont}
If $q \in L^1([0,T];H^{r,0})$ and $g \in H^{r,0}$, then for all $t\in [0,T]$, 
 \begin{gather}\label{eqn:PN_l2_stability_cont}
     |\psi|_{H^{r,0}}(t)\leq | g|_{H^{r,0}}+| q|_{L^1([0,t];H^{r,0})}.
 \end{gather}
 If, in addition, $q \in L^1([0,T];H^{i,j})$ and $g \in H^{i,j}$ for each $i,j$ such that $0 \leq i \leq r$, $0 \leq j \leq s$,  and $i+ j = r+s$, 
\begin{equation}\label{eqn:hrs_stability_continuum}
\begin{split}
|\psi|_{H^{r,s}}(t)
   &
        \leq C_s s!  \cA_{0}^s\left[|g|_{H^{r+s,0}}+|q|_{L^1([0,\bullet];H^{r+s,0})} \right](t)\\
    &+ C_s\sum_{i=0}^{s-1}\frac{s!}{(s-i)!} \cA_{0}^i[ F_0|g|_{H^{r+i,s-i}}](t) 
        +C_s\e\sum_{i=0}^{s-1}\frac{s!}{(s-i)!} \cA_{0}^{i+1}[|q|_{H^{r+i,s-i}}](t),
\end{split}\end{equation}
where $C_s$ is a constant depending only on s.
 \end{lem}
\begin{corl}[Isotropic data and zero initial condition for the continuous systems]\label{cor:isotropic_data_est_con}
Let $s\geq 1$. In the special case that $g=0$ and $q$ is isotropic, then if $t\in [0,T]$, 
\begin{gather}
    |\psi|_{H^{r,s}}(t)\leq C_ss!\cA_{0}^s[|q|_{L^1([0,\bullet];H^{r+s,0})}](t) .
\end{gather}
\end{corl}
 \begin{proof}[Proof of Lemma \ref{lem:stability_higher_order_cont}]
 With the obvious changes, the proof is the same line by line as proof in Lemma \ref{lem:stability_higher_order}, with some key steps replace by their continuous counterparts, namely, in the initial step we use \eqref{eq: regular-L2} instead of \eqref{eq: PN-L2} and we invoke Lemma \ref{lem:semi_norm_rec_cont} instead of Lemma \ref{lem:semi_norm_rec}. 
\end{proof}

\subsection{\texorpdfstring{$\Peqn$}{PN} error analysis \label{sec:pn_error}}
In this section we will analyze the error produced by the solution of \eqref{eq:transp_simple} when the $\Peqn$ approximation is used.
\begin{defn}[$\text{P}_N$ Error]
The $\text{\emph{P}}_N$ \emph{error} is
\begin{equation}
    e^N(t)=\psi(t)-\psi^N(t) = \eta^N(t) + \xi^N(t), 
\end{equation}
where $\eta^N=\psi-\mathcal{P}_N\psi$ is the \emph{consistency error} and $\xi^N=\cP_N\psi-\psi^N$ is the \emph{stability error.}  
\end{defn}
\begin{lem} 
\label{lem:stability_first_estimate}
For all $t\in[0,T]$, $\xi^N$ is controlled by $\eta^N$ via the following estimate: 
\begin{equation}
    \norma{\xi^N}_{L^2(X\times\sphe)}(t)\leq 
    \frac{1}{\e}\int_{0}^t\norma{\cP_N ( \Omega \cdot \nabla_x \eta^N}_{L^2(X\times\sphe)}(\tau)\,d\tau.
\end{equation}
\end{lem}

\begin{proof}
Applying the projection $\cP_N$ to \eqref{eq:transp_simple} and subtracting \eqref{eqn:pn_approx} yields a $\Peqn$ equation for $\xi^N$ with a source that depends on $\eta^N$:
    \begin{equation}
    \label{eq:PN_xi}         \e\partial_t\xi^N+\cP_N(\Omega\cdot\nabla_x\xi^N)+\frac{\sigma}{\e}\xi^N=\frac{\sigma}{\e}\overline{\xi^N}-\cP_N(\Omega\cdot\nabla_x\eta^N),\qquad
              \left.\xi^N\right|_{t=0}=0.
    \end{equation}
Thus \eqref{eq:PN_xi} follows immediately from the bound \eqref{eq:PN_L2_bound}, replacing $g$ by zero and $q$ by ${\veps^{-1} \cP_N(\Omega\cdot\nabla_x\eta^N)}$.
\end{proof}

\begin{lem}\label{lem:proj_error_control} Let $t\in [\alpha,\beta)\subseteq [0,T]$, then 
\begin{equation}
    \norma{\eta^N}_{L^2(X\times\sphe)}(t)\leq e^{-\frac{\sigma(t-\alpha)}{\e^2}}\norma{\eta^N}_{L^2(X\times\sphe)}(\alpha)+\e\cA_{\alpha}[\norma{\tP_N q}_{L^2(X\times\sphe)}](t)+\cA_{\alpha}[\norma{\tP_N(\Omega\cdot\nabla_{x}\psi)}_{L^2(X\times\sphe)}](t) \label{eq:great_estimate}
\end{equation}

\end{lem}
\begin{proof}
Applying the projection $\cP_N$ to \eqref{eq:transp_simple}, and subtracting it from \eqref{eq:transp_simple}, we see that $\eta^N$ satisfies
\begin{equation}
    \partial_t\eta^N-\frac{1}{\e}\cP_N(\Omega\cdot\nabla_x\eta^N)+\frac{\sigma}{\e^2}\eta^N=\tP_Nq-\frac{1}{\e}\left[\Omega\cdot\nabla_x\psi-\cP_N(\Omega\cdot\nabla_x\cP_N\psi)\right]
\end{equation}
Since for any $\phi \in L^2(\bbS^2)$, $(\cP_N \phi, \eta^N)_{L^2(\bbS^2)}=0$, testing the equation above against $\eta^N$ gives
\begin{equation}
\begin{split}
\frac{1}{2}\frac{d}{dt}\norma{\eta^N}_{L^2(X\times\sphe)}^2+\frac{\sigma}{\e^2}\norma{\eta^N}_{L^2(X\times\sphe)}^2&= (\tP_N q,\eta^N)_{L^2(X\times\sphe)}-\frac{1}{\e}(\Omega\cdot\nabla_x\psi,\eta^N)_{L^2(X\times\sphe)}\\
                                                                        &= (\tP_N q,\eta^N)_{L^2(X\times\sphe)}-\frac{1}{\e}(\tP_N(\Omega\cdot\nabla_x\psi),\eta^N)_{L^2(X\times\sphe)}\\
                                                                        &\leq (\norma{\tP_N q}_{L^2(X\times\sphe)}+\frac{1}{\e}\norma{\tP_N(\Omega\cdot\nabla_x\psi)}_{L^2(X\times\sphe)})\norma{\eta^N}_{L^2(X\times\sphe)}
                                                                        \end{split}\end{equation}
The conclusion then follows from Lemma \ref{lem:basic_ode_result}.                                                                  
\end{proof}
An immediate corollary of Lemma \ref{lem:proj_error_control} is the following:
\begin{lem}\label{lem:pn_proj_error_est} Let $s\geq 0$ and $N\geq\max\{0,s-1\}$. If $g\in H^{0,s}$ and $q\in L^{\infty}([0,T];H^{0,s})$. Then we have
    \begin{equation}
    \norma{\eta^N}_{L^2(X\times\sphe)}(T)\leq \frac{e^{-\sigma T/\e^2}}{(N+1)^s}|g|_{H^{0,s}}+\frac{\e}{(N+1)^s}|q|_{L^{\infty}([0,T];H^{0,s})}\cA_{0}[\mathbbm{1}](T)+\frac{T}{\e(N+1)^s}\sup_{\tau\in [0,T]}|\psi|_{H^{1,s}}(\tau).
\end{equation}
\end{lem}
\begin{proof}
    We apply Lemma \ref{lem:proj_error_control} with $\alpha=0$.  In this case $\norma{\eta^N}_{L^2(X\times\sphe)}(\alpha) = \norma{\tP_Ng}$.  Meanwhile, by Lemma \ref{lem:approx},  
    \begin{equation}
        \norma{\tP_N g}_{L^2(X\times\sphe)}\leq \frac{1}{(N+1)^s}|g|_{H^{0,s}}
\quand
        \norma{\tP_N q}_{L^2(X\times\sphe)}\leq \frac{1}{(N+1)^s}|q|_{H^{0,s}}.
    \end{equation}
  Thus since  $\cA_{\alpha}[f](t) \leq \e^{-1} (t -\alpha) \sup_{\tau\in[\alpha,t]} f(\tau)$
    (cf. \eqref{eq:A_one} with $k=1$), another application of Lemma \ref{lem:approx} gives
\begin{equation}
    \begin{split}
&\cA_0[\norma{\tP_N(\Omega\cdot\nabla_{x}\psi)}_{L^2(X\times\sphe)}](T)
        \leq \frac{T}{\e} \sup_{\tau\in[0,T]} \norma{\tP_N(\Omega\cdot\nabla_{x}\psi)}_{L^2}(\tau)\\
        & \qquad \qquad \leq\frac{T}{\e(N+1)^s}\sup_{\tau\in[0,T]}\,|\Omega\cdot\nabla_x\psi|_{H^{0,s}}(\tau)
        \leq\frac{T}{\e(N+1)^s}\sup_{\tau\in[0,T]}|\psi|_{H^{1,s}}(\tau)
    \end{split}
\end{equation}
Plugging the preceding bound into Lemma \ref{lem:proj_error_control} gives the result. 
\end{proof}
\begin{lem}[a priori estimate]\label{lem:eta_bound}
Let $s\geq 1$ and $N\geq s-1$.  Let $q \in L^\infty([0,T];H^{i,j})$ and $g \in H^{i,j}$ for each $i,j$ such that $0 \leq i \leq r$, $0 \leq j \leq s$,  and $i+ j = r+s$.  Then for all $t\in[0,T]$,

\begin{equation}
    \begin{split}
    |\psi|_{H^{r,s}}(t)& 
  \leq C_s\Bigg\{\left[|g|_{H^{r+s,}}+t|q|_{L^\infty([0,t);H^{r+s,0})} \right]\min\left(\frac{\e^ss!}{\sigma^s},\left(\frac{t}{\e}\right)^s\right) 
  +e^{-\sigma t/\e^2}\sum_{i=0}^{s-1}|g|_{H^{r+i,s-i}}\binom{s}{i}\frac{t^i}{\e^i}
    \\
       & \quad +\e\sum_{i=0}^{s-1}| q|_{L^{\infty}([0,t];H^{r+i,s-i})}\frac{s!}{(s-i)!}\min\left(\frac{\e^{i+1}}{\sigma^{i+1}},\frac{1}{(i+1)!}\left(\frac{t}{\e}\right)^{i+1}\right)\Bigg\}.
    \end{split}
\end{equation}
\end{lem}
\begin{proof}
Recall the stability estimate from Lemma \ref{lem:stability_higher_order_cont}:
\begin{equation}
    \begin{split}
    |\psi|_{H^{r,s}}(t)&\leq C_s s!  \cA_{0}^s\left[|g|_{H^{r+s,0}}+| q|_{L^1([0,\bullet];H^{r+s,0})} \right](t)\\
    &+ C_s\sum_{i=0}^{s-1}\frac{s!}{(s-i)!} \cA_{0}^i[ F_0|g|_{H^{r+i,s-i}}](t) 
        +C_s\e\sum_{i=0}^{s-1}\frac{s!}{(s-i)!} \cA_{0}^{i+1}[| q|_{H^{r+i,s-i}}](t).
    \end{split}
\end{equation}
Substituting the bounds for $\mathcal{A}_0[\mathbbm{1}]$ and the formula for $\cA_0[F_{\alpha}]$ from Lemma \ref{lem:A_estimates} into the above estimate yields the stated result.
\end{proof}
\begin{thm}[$\Peqn$ error]\label{thm:PN_error_est}
Let $s\geq 1$ and  $N \geq s-1$.  Let $q \in L^\infty([0,T];H^{i,j})$ and $g \in H^{i,j}$ for each $i,j$ such that $0 \leq j \leq s$, $i+ j \leq s+1$.  Then 
\begin{equation}\label{eq:pn_error_est} 
\begin{split}
  \norma{e^{N}}_{L^2(X\times\sphe)}(T)
  &\leq \frac{e^{-\sigma T/\e^2}}{(N+1)^s}|g|_{H^{0,s}}+\frac{1}{(N+1)^s}|q|_{L^{\infty}([0,T];H^{0,s})}\min\left(\frac{\e^2}{\sigma},T\right)\\
  &+ \frac{2C_s}{(N+1)^s}\Bigg\{ \Big(|g|_{H^{s+1,0}}+T|q|_{L^\infty([0,T];H^{s+1,0})} \Big)\min\left(\frac{\e^{s-1} s!T}{\sigma^s},\left(\frac{T}{\e}\right)^{s+1}\right)
        \\
    & \quad +e^{-\sigma T/\e^2}\sum_{i=0}^{s-1}|g|_{H^{1+i,s-i}}\binom{s}{i}\frac{T^{i+1}}{\e^{i+1}}\\
       & \quad + \sum_{i=0}^{s-1}|q|_{L^{\infty}([0,T];H^{1+i,s-i})}\frac{s!}{(s-i)!}\min\left(\frac{\e^{i+1}T}{\sigma^{i+1}},\frac{1}{(i+1)!}\frac{T^{i+2}}{\e^{i+1}}\right)\Bigg\}
\end{split}
\end{equation}
\end{thm}
\begin{proof}By the triangle inequality, Lemma \ref{lem:stability_first_estimate}, 
    \begin{equation}
    \begin{split}
          \norma{e^{N}}_{L^2(X\times\sphe)}(T) &\leq \|\eta^N\|_{L^2(X\times\sphe)}(T)+ \|\xi^N\|_{L^2(X\times\sphe)}(T)\\
          &\leq \|\eta^N\|_{L^2(X\times\sphe)}(T)+ \frac{1}{\e}\int_{0}^T\norma{\cP_N ( \Omega \cdot \nabla_x \eta^N)}_{L^2(X\times\sphe)}(\tau)\,d\tau\\
         &\leq \|\eta^N\|_{L^2(X\times\sphe)}(T)+ \frac{1}{\e}\int_{0}^T\norma{\nabla_x \eta^N}_{L^2(X\times\sphe)}(\tau)\,d\tau\\
          &\leq \|\eta^N\|_{L^2(X\times\sphe)}(T)+ \frac{1}{(N+1)^s}\frac{T}{\e} \sup_{\tau \in [0,T]}|\psi|_{H^{1,s}}(\tau)\\
          &\leq \frac{e^{-\sigma T/\e^2}}{(N+1)^s}|g|_{H^{0,s}}+\frac{\e}{(N+1)^s}|q|_{L^{\infty}([0,T];H^{0,s})}\cA_{0}[\mathbbm{1}](T)+\frac{2T}{\e(N+1)^s}\sup_{\tau\in [0,T]}|\psi|_{H^{1,s}}(\tau)
    \end{split}
    \end{equation}
In the last two lines, we applied spectral estimate in Lemma \ref{lem:approx}.  In the last line we used Lemma \ref{lem:pn_proj_error_est}. Applying Lemma \ref{lem:eta_bound} with $r=1$ yields the result.
\end{proof}
\begin{corl}[$\Peqn$ error for Isotropic data]\label{corl:PN_error_est_isotropic} Let $s\geq 1$ and  $N \geq s-1$.  Let $q \in L^\infty([0,T];H^{s+1,0})$ and $g \in H^{s+1,0}$. If $g$ and $q$ are isotropic, then
\begin{equation}\label{eq:iso_pn_error_est} 
  \norma{e^{N}}_{L^2(X\times\sphe)}(T)\leq\frac{2C_s}{(N+1)^s}\Big(|g|_{H^{s+1,0}}+T|q|_{L^\infty([0,T];H^{s+1,0})} \Big)\min\left(\frac{\e^{s-1} s!T}{\sigma^s},\left(\frac{T}{\e}\right)^{s+1}\right)
\end{equation}
\end{corl}
\section{Hybrid error analysis}

\subsection{A priori estimates of the uncollided component}\label{sec:hybrid_stability}
Since our goal is to derive error estimates which only depend on data, and $\psium[m]$ appears as a source in the collided equation, we require the following a priori estimates on $\psium[m]$ to bound $|\psicm[m]|_{H^{1,s}}$ in the proof of Theorem \ref{thm:hyb_final_err_est}.
\begin{lem}[Stability of the uncollided component\label{lem:one_interval_stab}]
Let $1\leq m\leq M$, $q \in L^\infty([0,T];H^{r,0})$, and $g \in H^{r,0}$ for some $r \geq 0$. Then for all $t\in [t_{m-1},t_{m})$,
\begin{subequations}
\begin{align} \label{eqn:L1_bd_t}
&|\psium[m]|_{H^{r,0}}(t)\leq e^{-\sigma(t-t_{m-1})/\e^2}|g|_{H^{r,0}} + e^{-\sigma(t-t_{m-1})/\e^2}|q|_{L^1([0,t_{m-1}];H^{r,0})}  + \e\cA_{t_{m-1}}[|q|_{H^{r,0}}](t), 
\\
\label{eqn:L1_bd_tmp1}
&|\psium[m]|_{H^{r,0}}(t_{m}) +  \frac{\sigma}{\e^2}|\psium[m]|_{L^1([t_{m-1},t_{m}];H^{r,0})} \leq |g|_{H^{r,0}} +  |q|_{L^1([0,t_{m}];H^{r,0})} , \\
&|\psium[m]|_{L^1([t_{m-1},t];H^{r,0})}(t)\leq \e(|g|_{H^{r,0}}  + t|q|_{L^{\infty}([0,t];H^{r,0})})\cA_{t_{m-1}}[\mathbbm{1}](t).\label{eqn:Linf_bd}
\end{align}
\end{subequations}
\end{lem}
\begin{proof} We prove the result only for $r = 0$ since the other cases are obtained by applying the same techniques to \eqref{eq:hybrid_intervals_u} differentiated $r$ times in space. 
Testing \eqref{eq:hybrid_intervals_u} with $\psium[m]$ and applying Cauchy-Schwarz inequality, we obtain the following differential inequality
\begin{equation}
    \frac{1}{2}\frac{d}{dt}\norma{\psium[m]}^2+\frac{\sigma}{\e^2}\norma{\psium[m]}^2\leq \norma{q}\norma{\psium[m]}\label{eq:unc_diff_inq}.
\end{equation}
 An application of \eqref{eqn:basic_ode_result} in Lemma \ref{lem:basic_ode_result}, and the fact that $\psium[m](t_{m-1})=\psi(t_{m-1})$ gives
\begin{gather}
\norma{\psium[m]}_{L^2(X\times\sphe)}(t)\leq e^{-\sigma(t-t_{m-1})/\e^2}\norma{\psi}_{L^2(X\times\sphe)} (t_{m-1}) + \e\cA_{t_{m-1}}[\|q\|_{L^2(X\times\sphe)}](t). \label{eqn:stab_unc_temp}
\end{gather}
Using Lemma \ref{lem:stability_higher_order_cont} on $\norma{\psi}_{L^2(X\times\sphe)}(t_{m-1})$ gives the first result \eqref{eqn:L1_bd_t}. An application of \eqref{eqn:l1_ode_result} in Lemma \ref{lem:basic_ode_result} over $[t_{m-1},t_{m})$ gives
\begin{equation}
    \norma{\psium[m]}_{L^2(X\times\sphe)}(t_{m}) + \frac{\sigma}{\e^2}\|\psium[m]\|_{L^1([t_{m-1},t_{m}],L^2(X\times\sphe))}\leq \|q\|_{L^1([t_{m-1},t_{m}],L^2(X\times\sphe))}+\norma{\psi}_{L^2(X\times\sphe)}(t_{m-1}),
\end{equation}
and then another application of Lemma \ref{lem:stability_higher_order_cont} to bound $\norma{\psi}_{L^2(X\times\sphe)}(t_{m-1})$ gives \eqref{eqn:L1_bd_tmp1}.
Estimate \eqref{eqn:Linf_bd} is obtained from integrating the first estimate \eqref{eqn:L1_bd_t} from $t_{m-1}$ to $t$.
\end{proof}



\subsection{Hybrid error estimates}\label{sec:pn_hybrid_equation}
In this section we will analyze the error in the hybrid method using the formulation in  \eqref{eq:hybrid_intervals_PN}.
\begin{defn}[Hybrid errors\label{def:errors}] Let $1\leq m\leq M$ and $t\in[t_{m-1},t_{m})$. The $m$-th \emph{hybrid error} is 
\begin{equation}
    e_{m}^{N}(t)=e_{\mathrm{u},m}^{N}(t)+e_{\mathrm{c},m}^{N}(t),\label{eq:total_error}
\end{equation}
where
$\eum^{N}(t) = \psium(t) - \psium^{N}(t)$ and
$\ecm^{N}(t) = \psicm(t) - \psicm^{N}(t)$
are the $m$-th errors in the uncollided and collided components. The collided error can be further decomposed as
\begin{equation}
\etacm^{N}(t) = \psicm(t)-\cP_{N}\psicm(t)  
\qquand 
\xicm^{N} = \cP_{N}\psicm(t) - \psicm ^{N}(t)
\end{equation}
so that 
$
\errcol{m}(t) = \eta_{\mathrm{c},m}^N(t) + \xi_{\mathrm{c},m}^N(t).
$
Here $\eta^N_{\mathrm{c},m}$ is the \emph{ $m$-th collided consistency error} and $\xi^N_{\mathrm{c},m}$ is the \emph{$m$-th collided stability error}. The error $e^{N}_{M}$ is simply called the \emph{hybrid error}.
\end{defn}
This next lemma gives a one-step analysis of the growth of the error in the uncollided and collided components from $t_{m-1}$ to $t_{m}$.
\begin{lem}\label{lem:unc_err_recu}
Let $1\leq m\leq M$, then if $t\in[t_{m-1},t_{m})$, the $m$-th uncollided and collided errors satisfy, respectively,
\begin{subequations}
\begin{align}
    \norma{e_{\mathrm{u},m}^N}_{L^2(X\times\sphe)}(t)
    &\leq e^{-\sigma(t-t_{m-1)}/\e^2}\norma{e_{{m-1}}^{N}}_{L^2(X\times\sphe)}(t^-_{m-1}),
    \label{eq:unc_err_recu}\\
       \norma{e_{\mathrm{c},m}^N}_{L^2(X\times\sphe)}(t)
       &\leq \left(1-e^{-\sigma(t-t_{m-1)}/\e^2} \right)\norma{e_{{m-1}}^{N}}_{L^2(X\times\sphe)}(t^-_{m-1}) \nonumber\\
       &\quad 
       +\norma{\eta_{\mathrm{c},m}^N}_{L^2(X\times\sphe)}(t)
      +\frac{1}{\e}\norma{\Omega \cdot\nabla_x\eta_{\mathrm{c},m}^N}_{L^1([t_{m-1},t];L^{2}(X\times\sphe))}. \label{eq:col_err_recu}
\end{align}
\end{subequations}
\end{lem}
\begin{proof}
Subtracting \eqref{eq:hybrid_intervals_PN_uncollided} from \eqref{eq:hybrid_intervals_u}, yields the following evolution equation for $e^{N}_{\mathrm{u},m}$, 
\begin{align}
\e\partial_t e^{N}_{\mathrm{u},m}+\Omega\cdot\nabla_x e^{N}_{\mathrm{u},m}+\frac{\sigma}{\e} e^{N}_{\mathrm{u},m}=0,\qquad
\left.e^{N}_{\mathrm{u},m}\right|_{t=t_{m-1}}=
         e^{N}_{m-1}(t_{m-1}^-),
\end{align}
where $e^{N}_{0}(t_0^-) = 0$.
Thus applying \eqref{eqn:L1_bd_t} from Lemma \ref{lem:one_interval_stab}, with $r=0$ and a zero source term yields \eqref{eq:unc_err_recu}. To prove \eqref{eq:col_err_recu}, we subtract from \eqref{eq:hybrid_intervals_PN_collided}  the projection applied to \eqref{eq:hybrid_intervals_c}.  This gives the following $\Peqn$ equation for $\xi_{\mathrm{c},m}^N$ 
\begin{subequations}
\begin{align}
\label{eqn:xi_c}
    &\e\partial_t\xi_{\mathrm{c},m}^N
    +\cP_N(\Omega\cdot\nabla_x\xi_{\mathrm{c},m}^N)
    +\dfrac{\sigma}{\e}\xi_{\mathrm{c},m}^N
    =\dfrac{\sigma}{\e}(
    \overline{\xi_{\mathrm{c},m}^N}+
    \overline{e_{\mathrm{u},m}^N}
    )
    -\cP_N(\Omega\cdot\nabla_x\eta_{\mathrm{c},m}^N),\\
    &\left.\xi_{\mathrm{c},m}^N\right|_{t=t_{m-1}}=0.
\end{align}
\end{subequations}
We apply Lemma \ref{lem:stability_higher_order} to \eqref{eqn:xi_c} with zero initial data and source $ \e^{-2}\sigma\overline{e_{\mathrm{u},m}^N}-\e^{-1}\cP_N(\Omega\cdot\nabla_x\eta_{\mathrm{c},m}^N)$. Combined with bound \eqref{eq:unc_err_recu}, the estimate on $\xi_{\mathrm{c},m}^N$ becomes
\begin{equation}
\begin{split}
    \norma{\xi_{\mathrm{c},m}^N}_{L^2(X\times\sphe)}(t)
    &\leq \frac{\sigma}{\e^2}\int_{t_{m-1}}^t\norma{e_{\mathrm{u},m}^N}_{L^2(X\times\sphe)}(\tau)\,d\tau+\frac{1}{\e}\int_{t_{m-1}}^t\norma{\Omega \cdot\nabla_x\eta_{\mathrm{c},m}^N}_{L^2(X\times\sphe)}(\tau)\,d\tau\\
    &\leq\frac{\sigma}{\e^2}\norma{e_{{m-1}}^{N}}(t_{m-1}^-)\int_{t_{m-1}}^te^{-\sigma(\tau-t_{m-1})/\e^2}\,d\tau+\frac{1}{\e}\int_{t_{m-1}}^t\norma{\Omega \cdot\nabla_x\eta_{\mathrm{c},m}^N}_{L^2(X\times\sphe)}(\tau)\,d\tau\\ 
    &=\left(1-e^{-\sigma(t-t_{m-1)}/\e^2}\right)\norma{e_{{m-1}}^{N}}_{L^2(X\times\sphe)}(t_{m-1}^-)+\frac{1}{\e}\int_{t_{m-1}}^t\norma{\Omega \cdot\nabla_x\eta_{\mathrm{c},m}^N}_{L^2(X\times\sphe)}(\tau)\,d\tau.
\end{split}
\end{equation}
Adding $\norma{\eta_{\mathrm{c},m}^N}_{L^2(X\times\sphe)}(t)$ to the both sides recovers  \eqref{eq:col_err_recu}.\end{proof}

Now we can state an error for all time that only depends on the approximation properties of the spherical harmonic discretization on the solution.
\begin{thm}\label{thm:hybrid_error_general} The hybrid error $e_{M}^N$ satisfies 
\begin{gather}
 \norma{e_{M}^{N}}_{L^2(X\times\sphe)}(t_{M}^-)
 \leq \sum_{m=1}^{M}
 \left(\norma{\eta_{\mathrm{c},m}^{N}}_{L^2(X\times\sphe)}(t_{m}^-)
+\frac{1}{\e}\norma{\Omega \cdot\nabla_x\eta_{\mathrm{c},m}^{N}}_{L^1([t_{m-1},t_{m}];L^{2}(X\times\sphe))}\right).
\end{gather}
\end{thm}
\begin{proof}
Adding the inequalities in Lemma \ref{lem:unc_err_recu} 
and taking the limit $t \to t_{m+1}^-$ gives
\begin{equation}
    \begin{split}
 \norma{e_{M}^{N}}_{L^2(X\times\sphe)}(t^-_{M})
 \leq \norma{e_{{M-1}}^{N}}_{L^2(X\times\sphe)}(t^-_{M-1})
 &+\norma{\eta_{\mathrm{c},{M}}^{N}}_{L^2(X\times\sphe)}(t^-_{M})\\
&+\frac{1}{\e}\norma{\Omega \cdot\nabla_x\eta_{\mathrm{c},{M}}^{N}}_{L^1([t_{M-1},t_{M}];L^{2}(X\times\sphe))}
    \end{split}
    \end{equation}
Exhausting this recursion until $e_{0}^{N}(t^-_{0}) = 0$ yields the result. 
\end{proof}

\begin{lem}\label{lem:hybrid_consistency _estimate} Let $s\geq 0$, $N \geq\max\{0,s-1\}$, and $1\leq m\leq M$.  Then the $m$-th projection error $\eta^N_{\mathrm{c},m}$ satisfies,
\begin{equation}
\norma{\eta_{\mathrm{c},m}^N}_{L^2(X\times\sphe)}(t^-_{m})\leq \frac{\Delta t}{\e(N+1)^s}\sup_{\tau\in[t_{m-1},t_{m}]}\left|\psi_{\mathrm{c},m}\right|_{H^{1,s}}(\tau).
\end{equation}
\end{lem}

\begin{proof}
    An application of Lemma \ref{lem:proj_error_control} with $\alpha=t_{m-1}$, $\beta=t_{m}$, $\eta_{\mathrm{c},m}^N(t_{m-1})=0$, an isotropic source $q=\frac{\sigma}{\e^2}\overline{\psium[m]}$, along with the fact that $\cA_{\alpha}[f](t) \leq \e^{-1} (t -\alpha) \sup_{\tau\in[\alpha,t]} f(\tau)$
    (cf. \eqref{eq:A_one} with $k=1$)) and the spectral estimate in Lemma \ref{lem:approx}, gives
    \begin{equation}
    \begin{split}
    &\norma{\eta_{\mathrm{c},m}^N}_{L^2(X \times \bbS^2)}(t_{m}^-) 
    \leq \frac{\dt}{\e}\sup_{\tau \in [t_{m-1},t_m]}\norma{\tP_N(\Omega\cdot\nabla_{x}\psicm[m]))}_{L^2(X \times \bbS^2)}(\tau)\,d\tau\\
    & \quad \leq \frac{\dt}{\e(N+1)^s}\sup_{\tau \in [t_{m-1},t_m]}
    \left|\Omega\cdot\nabla_x\psi_{\mathrm{c},m}\right|_{H^{0,s}}(\tau)\,d\tau
     \leq \frac{\Delta t}{\e(N+1)^s}\sup_{\tau \in [t_{m-1},t_m]}\left|\psi_{\mathrm{c},m}\right|_{H^{1,s}}(\tau).
     \end{split}
    \end{equation}
    \end{proof}
 \subsection{Estimating hybrid error in terms of the data}

Finally, we will apply the approximation properties and stability estimates to the estimate in Theorem \ref{thm:hybrid_error_general} to obtain an estimate that depends only on the regularity of the data.

\begin{thm}\label{thm:hyb_final_err_est}
Let $s\geq 1$ and $N \geq s-1$.  If $q \in L^1([0,T];H^{s+1,0})$ and $g \in H^{s+1,0}$, then
\begin{gather}\label{eqn:hyb_final_err_est}
           \norma{e_{M}^{N}}_{L^2(X\times\sphe)}(T^-)\leq \frac{2C_s}{(N+1)^s} \left( \left|g\right|_{H^{s+1,0}} +  T\left|q\right|_{L^{\infty}([0,T];H^{s+1,0})}\right)\min\left(\frac{\e^{s-1}s! T}{\sigma^s}, \frac{\Delta t^{s}T}{\e^{s+1}}\min\left(1,\frac{\Delta t \sigma}{\e^2}\right)\right).
\end{gather}
\end{thm}

\begin{proof}
Using $\|\Omega \cdot \nabla_x \psicm[m]\|_{L^2(X\times\sphe)} \leq \|\nabla_x \psicm[m]\|_{L^2(X\times\sphe)}$, and applying Lemma \ref{lem:approx} and Lemma \ref{lem:hybrid_consistency _estimate} to Theorem \ref{thm:hybrid_error_general} yields, for $N \geq s-1$,
\begin{align*}
   \norma{e_{M}^{N}}_{L^2(X\times\sphe)}(t^{-}_{M})
   &\leq \frac{1}{(N+1)^s}\sum_{m=1}^{M}
   \left(\frac{\Delta t}{\e}\sup_{\tau\in[t_{m-1},t_m]}\left|\psi_{\mathrm{c},m}\right|_{H^{1,s}}(\tau)+\frac{1}{\e}|\psicm[m]|_{L^1([t_{m-1},t_m];H^{1,s})} \right)
       \\
    &\leq \frac{2\Delta t}{\e (N+1)^s} \sum_{m=1}^{M}\sup_{\tau\in [t_{m-1},t_m]}|\psicm[m]|_{H^{1,s}}(\tau)
\end{align*}
Applying Corollary \ref{cor:isotropic_data_est_con} with $q = \frac{\sigma}{\e^2}\overline{\psium[m]}$ and $r = 1$,  
\begin{equation}\label{eqn:hyb_error_before_est}
           \norma{e_{M}^{N}(t^-_{M})}_{L^2(X\times\sphe)}\leq 2C_ss!\frac{\sigma}{\e^3}\frac{\Delta t}{(N+1)^s} \sum_{m=1}^{M} \sup_{\tau\in [t_{m-1},t_m]}\cA_{t_{m-1}}^s\left[|\psium[m]|_{L^1([t_{m-1},\bullet];H^{s+1,0})}\right](\tau).
\end{equation}
For the summand above, it follows from \eqref{eqn:L1_bd_tmp1}, the monotonicity of $\cA_{\alpha}$, and Lemma \ref{lem:A_estimates} that
\begin{equation}
\begin{split}
   \sup_{\tau\in [t_{m-1},t_m]} \cA_{t_{m-1}}^s\left[|\psium[m]|_{L^1([t_{m-1},\bullet];H^{s+1,0})}\right](\tau) &\leq \sup_{\tau\in [t_{m-1},t_m]}|\psium[m]|_{L^1([t_{m-1},t_m];H^{s+1,0})}\cA_{t_{m-1}}^s\left[\mathbbm{1}\right](\tau)
    \\
    & \hspace{-50pt}\leq \frac{\e^2}{\sigma}\left( \left|g\right|_{H^{s+1,0}} +  \left|q\right|_{L^1([0,T];H^{s+1,0})}\right)\cA_{t_{m-1}}^s\left[\mathbbm{1}\right](t_{m})
    \\
    &\hspace{-50pt} \leq 
    \frac{\e^2}{\sigma}\left( \left|g\right|_{H^{s+1,0}} +  \left|q\right|_{L^1([0,T];H^{s+1,0})}\right)\min\left(\frac{\e^s}{\sigma^s}, \frac{1}{s!}\left(\frac{\Delta t}{\e}\right)^s\right).
\end{split}           
\end{equation}
Pluggin the above bound into \eqref{eqn:hyb_error_before_est} yields (since $T = M \Delta t$)
\begin{equation}\label{eqn:hyb_est_one}
\begin{split}
           \norma{e_{M}^{N}}_{L^2(X\times\sphe)}(t^-_{M})
           &\leq 2C_s s!\frac{\Delta t M }{\e(N+1)^s} \left( \left|g\right|_{H^{s+1,0}} +  \left|q\right|_{L^1([0,T];H^{s+1,0})}\right)
           \min\left(\frac{\e^s}{\sigma^s}, \frac{\Delta t^{s}}{s! \e^s}\right)\\
           &= \frac{2C_s}{(N+1)^s} \left( \left|g\right|_{H^{s+1,0}} +  \left|q\right|_{L^1([0,T];H^{s+1,0})}\right)\min\left(\frac{\e^{s-1}s!T}{\sigma^s}, \frac{\Delta t^{s}T}{\e^{s+1}}\right).
\end{split}           
\end{equation}
On the other hand, it follows from \eqref{eqn:Linf_bd} that
\begin{equation}
\begin{split}
   &\sup_{\tau\in [t_{m-1},t_m]} \cA_{t_{m-1}}^s\left[|\psium[m]|_{L^1([t_{m-1},\bullet];H^{s+1,0})}\right](\tau) \\
   & \qquad \qquad \leq \e(\left|g\right|_{H^{s+1,0}}  + T\left|q\right|_{L^{\infty}([0,T];H^{s+1,0})})\sup_{\tau\in [t_{m-1},t_m]}\cA_{t_{m-1}}^{s+1}[\mathbbm{1}](\tau)
   \\
   & \qquad \qquad  \leq \e(\left|g\right|_{H^{s+1,0}}  + T\left|q\right|_{L^{\infty}([0,T];H^{s+1,0})})\min\left(\frac{\e^{s+1}}{\sigma^{s+1}}, \frac{1}{(s+1)!}\left(\frac{\Delta t}{\e}\right)^{s+1}\right).
\end{split}           
\end{equation}
Plugging this into \eqref{eqn:hyb_error_before_est} yields,
\begin{gather}\label{eqn:hyb_est_two}
           \norma{e_{M}^{N}}_{L^2(X\times\sphe)}(t^-_{M})\leq \frac{2C_s}{(N+1)^s} \left(\left|g\right|_{H^{s+1,0}}  + T\left|q\right|_{L^{\infty}([0,T);H^{s+1,0})}\right)\min\left(\frac{\e^{s-1}s!T}{\sigma^s}, \frac{\Delta t^{s+1}\sigma T}{\e^{s+3}}\right).
\end{gather}
Taking a minimum of the right hand sides of \eqref{eqn:hyb_est_one} and \eqref{eqn:hyb_est_two} yields the result.
\end{proof}
\section{Return to the original transport model}
In this section we will show error estimates for the model \eqref{eq:transp}. The analogous discretizations for the models are the following. For the non-splitting $\Peqn$ discretization we seek \mbox{$\Psi^{\e,N}\in C([0,T);X \times \bbP_{N}(\bbS^2))$}, satisfying 
\begin{subequations}
\begin{align}\label{eqn:pn_approx_original}
&    \e\partial_t\Psi^{\e,N}+\cP_N (\Omega\cdot\nabla_x\Psi^{\e,N})+\frac{\sigt}{\e}\Psi^{\e,N}=\left(\frac{\sigt}{\e}-\e\siga\right)\overline{\Psi^{\e,N}}+\e\cP_N Q,\\
& \Psi^{\e,N}|_{t=0} = \cP_N g
 \end{align}
  \end{subequations}
and for the hybrid, we seek $\Psi_m^{\e,N} = \Psi_{\mathrm{u},m}^{\e,N} + \Psi_{\mathrm{c},m}^{\e,N}$ where for each $m \in \{1,2,\dots,M\}$
\begin{equation}
    \left(\Psi_{\mathrm{u},m}^{\e,N}, \Psi_{\mathrm{c},m}^{\e,N} \right) \in C([t_{m-1},t_{m}); X \times L^2(\bbS^2)) \times C([t_{m-1},t_{m});X \times \bbP_{N}(\bbS^2)) 
\end{equation} 
satifies
\begin{subequations}
  \begin{align}    &\e\partial_t\Psi_{{\mathrm{u}},m}^{\e,N}+\Omega\cdot\nabla_x\Psi_{{\mathrm{u}},m}^{\e,N}+\frac{\sigt}{\e}\Psi_{{\mathrm{u}},m}^{\e,N}=\e Q,
      \label{eq:hybrid_intervals_PN_uncollided_transport}\\
    &\e\partial_t\Psi_{\mathrm{c},m}^{\e,N}+\cP_N\left(\Omega\cdot\nabla_x\Psi_{\mathrm{c},m}^{\e,N}\right)+\frac{\sigt}{\e}\Psi_{\mathrm{c},m}^{\e,N}=\left(\frac{\sigt}{\e}-\e\siga\right)\overline{\Psi_{\mathrm{u},m}^{\e,N}}+\overline{\Psi_{\mathrm{c},m}^{\e,N}},
     \label{eq:hybrid_intervals_PN_collided_transport}\\\\
    &\left.\Psi_{\mathrm{c},m}^{\e,N}\right|_{t=t_{m-1}}=0, \quad \left.\Psi_{\mathrm{u},m}^{\e,N}\right|_{t=t_{m-1}}=
   \begin{cases}
   g, &\quad m = 1,\\
   \Psi_{\mathrm{u},m-1}^{\e,N}(t^-_{m-1}) + \Psi_{\mathrm{c},m-1}^{\e,N}(t^-_{m-1}) &\quad m > 1.
   \end{cases}
    \label{eq:hybrid_intervals_PN_remap_transport}
    \end{align}
        \label{eq:hybrid_intervals_PN_transport}%
    \end{subequations}
We define the correspondent $\Peqn$ error and $m$-th hybrid error respectively as, 
\begin{equation*}
    e^{\e,N}=\Psi^{\e}-\Psi^{\e,N}\quand e^{\e,N}_m=\Psi_m^{\e}-\Psi^{\e,N}_m 
\end{equation*}
Since $\psi^N=e^{\siga t}\Psi^{\e,N}$ and $\psi_m^{N}=e^{\siga t}\Psi_m^{\e,N}$, applying Theorems \ref{thm:PN_error_est} and \ref{thm:hyb_final_err_est} gives the following estimate for the $\Peqn$ error for the original transport model \eqref{eq:transp}.
\begin{thm} \label{thm:PN_error_est_abs}
Let $s\geq 1$ and $N\geq s-1$, $Q \in L^\infty([0,T];H^{i,j})$ and $g \in H^{i,j}$ for each $i,j$ such that $0 \leq j \leq s$, $i+ j \leq s+1$. Then
\begin{equation}\label{eq:pn_error_est_abs} 
\begin{split}
  \norma{e^{\e,N}}_{L^2(X\times\sphe)}(T)&\leq \frac{e^{-(\sigt+\e^2\siga)T/\e^2}}{(N+1)^s}|g|_{H^{0,s}}+\frac{1}{(N+1)^s}|Q|_{L^{\infty}([0,T];H^{0,s})}\min\left(\frac{\e^2}{\sigt},T\right)\\
  +2C_s \frac{1}{(N+1)^s}&\Bigg( \left[e^{-\siga T}|g|_{H^{s+1,0}}+T|Q|_{L^\infty([0,T];H^{s+1,0})} \right]\min\left(\frac{\e^{s-1} s!T}{\sigt^s},\left(\frac{T}{\e}\right)^{s+1}\right)
        \\
    &\quad+e^{-(\sigt+\e^2\siga)T}\sum_{i=0}^{s-1}|g|_{H^{1+i,s-i}}\binom{s}{i}\frac{T^{i+1}}{\e^{i+1}}
    \\
       & +\sum_{i=0}^{s-1}|Q|_{L^{\infty}([0,T];H^{1+i,s-i})}\binom{s}{i}\min\left(\frac{\e^{i+1}T}{\sigt^{i+1}},\frac{1}{(i+1)!}\frac{T^{i+2}}{\e^{i+1}}\right)\Bigg)
\end{split}
\end{equation}
\end{thm}
Meanwhile for the hybrid approximation.
\begin{thm}\label{thm:hyb_final_err_est_abs}
Let $s\geq 1$ and $N \geq s-1$, $Q \in L^1([0,T];H^{s+1,0})$ and $g \in H^{s+1,0}$, we have 
\begin{multline}
\label{eqn:hyb_final_err_est_abs}
           \norma{e_{M}^{\e,N}}_{L^2(X\times\sphe)}(T^-)\leq \frac{2C_s}{(N+1)^s} \left( e^{-\siga T}\left|g\right|_{H^{s+1,0}} +  T\left|Q\right|_{L^{\infty}([0,T];H^{s+1,0})}\right) \\
           \times \min\left(\frac{\e^{s-1}s! T}{\sigt^s}, \frac{\Delta t^{s}T}{\e^{s+1}}\min\left(1,\frac{\Delta t \sigt}{\e^2}\right)\right).
\end{multline}
\end{thm}
\section{Conclusion}
In this paper, we have derived multiscale error estimates for the $\Peqn$ approximation of the RTE and for a hybrid approximation for the RTE that is built using the $\Peqn$ approximation. By construction, the hybrid is is more expensive; we use these error estimates to understand the benefits of the additional expense for different parameter regimes.  At each time step in the hybrid approximation, the collided equation is equipped with isotropic initial conditions and zero initial condition.  In scattering dominating regimes, this property is key to improved estimates over the monolithic $\Peqn$ approach.  Meanwhile, in purely absorbing regimes, the hybrid captures the RTE solution exactly.

In the future work, we intend to revisit the current analysis for more general problems on non-periodic domains, with non-constant cross-sections and inflow boundary conditions.  In addition, we intend to explicitly examine the effects of angular discretization errors in the treatment of the uncollided equation, which for the purposes of the current paper was assumed to be solved exactly. 
\appendix
 \section{Other estimates}
\begin{lem}\label{lem:basic_ode_result} Assume $\chi$ is a non-negative continuous function on $[\alpha,\beta]$.  Assume $\phi\in C^1([\alpha,\beta])$, $\phi\geq 0$, and  satisfies the following differential inequality
\begin{equation}
    \frac{1}{2}(\phi(t)^2)'+\kappa \phi(t)^2\leq \chi(t)\phi(t), \quad \phi(\alpha)=\phi_{\alpha}\geq 0.\label{eqn:general_ode} 
\end{equation}
Then for all $t\in[\alpha,\beta]$, 
\begin{equation}
    \phi(t)+\kappa\int_{\alpha}^t\phi(\tau)\,d\tau \leq\phi_{\alpha}+\int_{\alpha}^t\chi(\tau)\,d\tau.\label{eqn:l1_ode_result} 
\end{equation}
Furthermore
\begin{equation}
    \phi(t)\leq e^{-\kappa(t-\alpha)}\phi_{\alpha}+\int_{\alpha}^te^{-\kappa(t-\tau)}\chi(\tau)\,d\tau.\label{eqn:basic_ode_result}
\end{equation}
\end{lem}
\begin{proof}
We prove first \eqref{eqn:l1_ode_result}. Since $\phi$ and $\chi$ are non-negative functions, it follows that for any arbitrary $\delta>0$, the  following differential inequality holds
\begin{equation}
\phi(t)\phi'(t)+\kappa \phi(t)^2\leq \chi(t)(\phi(t)+\delta),\end{equation}
dividing both sides of the inequality by $\phi+\delta$, and integrating in time, we arrive at 
\begin{align*}
    \phi(t)+\kappa\int_{\alpha}^t \phi(\tau)\,d\tau&\leq \phi_{\alpha}+\int_{\alpha}^t\chi(\tau)\,d\tau+\delta\ln\left|\frac{\phi(t)+\delta}{\phi_{\alpha}+\delta}\right|+\kappa\delta\int_{\alpha}^t \frac{\phi(\tau)}{\phi(\tau)+\delta}\,d\tau,\\
    &\leq \phi_{\alpha}+\int_{\alpha}^t\chi(\tau)\,d\tau+\delta\ln\left|\frac{\phi(t)+\delta}{\phi_{\alpha}+\delta}\right|+\kappa\delta(t-\alpha),
\end{align*}
the conclusion follows taking $\delta\to 0^+$.   

We next prove \eqref{eqn:basic_ode_result}. When $\kappa=0$ the result follows immediately from \eqref{eqn:l1_ode_result}:  
\begin{equation}
\label{eqn:kappa_zero}
    \phi(t)\leq \phi_{\alpha}+\int_{\alpha}^t\chi(\tau)\,d\tau.
\end{equation}
For the general case we multiply \eqref{eqn:general_ode} by $e^{2\kappa t}$, obtaining
\begin{equation}
    \frac{1}{2}\left[\Phi(t)^2\right]'\leq e^{\kappa t}\chi(t)\Phi(t), 
\end{equation}
where $\Phi(t)=e^{\kappa t}\phi(t)$.  Applying \eqref{eqn:kappa_zero} to $\Phi$ and undoing the transformation yields \eqref{eqn:basic_ode_result}.
\end{proof}


\end{document}